\documentclass{article}

\IfFileExists{neurips_2026.sty}{\usepackage[preprint]{neurips_2026}}{}
\makeatletter
\renewcommand{\@notice}{}
\makeatother
\usepackage[utf8]{inputenc}
\usepackage[T1]{fontenc}
\usepackage{microtype}
\usepackage{booktabs}
\usepackage{url}
\usepackage{amsmath,amssymb,amsthm,mathtools}
\usepackage{algorithm}
\usepackage{algpseudocode}
\usepackage{enumitem}
\usepackage{graphicx}
\graphicspath{{EXP/}{./}}
\usepackage{xcolor}
\usepackage[hidelinks]{hyperref}
\usepackage[nameinlink,capitalize,noabbrev]{cleveref}

\theoremstyle{plain}
\newtheorem{theorem}{Theorem}[section]
\newtheorem{lemma}[theorem]{Lemma}
\newtheorem{proposition}[theorem]{Proposition}
\newtheorem{corollary}[theorem]{Corollary}
\newtheorem{assumption}[theorem]{Assumption}

\theoremstyle{definition}

\newtheorem{remark}[theorem]{Remark}

\newcommand{\R}{\mathbb{R}}
\newcommand{\E}{\mathbb{E}}

\newcommand{\PP}{\mathbb{P}}
\newcommand{\Pb}{\mathbb{P}}
\newcommand{\ip}[2]{\left\langle #1,#2\right\rangle}
\newcommand{\norm}[1]{\left\|#1\right\|}
\newcommand{\abs}[1]{\left|#1\right|}

\newcommand{\Unif}{\operatorname{Unif}}
\newcommand{\Sph}{\mathbb{S}^{d-1}}
\newcommand{\Ball}{\mathbb{B}^{d}}

\newcommand{\wtO}{\widetilde{\mathcal{O}}}

\newcommand{\safeincludegraphics}[2][]{%
  \IfFileExists{#2}{%
    \includegraphics[#1]{#2}%
  }{%
    \fbox{\parbox[c][0.25\textheight][c]{0.9\linewidth}{%
      \centering Missing figure: \texttt{\detokenize{#2}}%
    }}%
  }%
}
\title{Stochastic Zeroth-Order Optimization Under Heavy-Tailed Noise}
\author{%
\begin{tabular}{c}
Taha El Bakkali$^{1}$ \quad Elmahdi Chayti$^{2}$ \quad Qiuyi Zhang$^{3}$\\
Imane Rahali$^{1}$ \quad Omar Saadi$^{1}$\\[0.8em]
{\normalfont\small $^{1}$College of Computing}\\
{\normalfont\small Mohammed VI Polytechnic University (UM6P)}\\[0.35em]
{\normalfont\small $^{2}$EPFL \qquad $^{3}$Elorian AI}
\end{tabular}%
}

\date{}

\begin{document}
\maketitle
\begin{abstract}
We study stochastic zeroth-order (ZO) optimization of smooth nonconvex
objectives under heavy-tailed sample-gradient noise. This regime is motivated
by empirical evidence that gradient noise in modern machine learning can violate
the bounded-variance assumptions used in classical ZO theory. While first-order
methods have optimal rates under bounded $p$-th moment noise for
$p\in(1,2]$, analogous high-probability guarantees for nonconvex ZO methods
are much less understood.

The ZO setting is not a direct corollary of first-order theory. First-order
methods observe stochastic gradients, whereas derivative-free methods only query
noisy function values and build finite-difference estimates. Thus, weak-$L_p$
control of $\nabla F(x;\xi)-\nabla f(x)$ must first be transferred to scalar
directional estimates.

We propose the Robust Scalar-Clipped Zeroth-Order method (RSC-ZO), a two-point
method that clips each scalar directional derivative before aggregation. Under
sample-wise smoothness and a weak-$L_p$ tail condition on the sample-gradient
noise, RSC-ZO finds an $\varepsilon$-stationary point with high probability
using
$$
\widetilde{O}\!\left(
d^{\frac{p}{2(p-1)}}\varepsilon^{-\frac{3p-2}{p-1}}
\right)
$$
noisy function evaluations. This matches the optimal first-order
$\varepsilon$-dependence. At $p=2$, the bound becomes
$\widetilde{O}(d\varepsilon^{-4})$, matching the classical stochastic ZO
dimension--accuracy dependence, but with a high-probability guarantee and under
a weaker weak-$L_2$ condition that can allow infinite variance. We also analyze
a momentum variant and quantify its batch-size/stepsize tradeoff.
\end{abstract}

\section{Introduction}
\label{sec:introduction}

We consider the stochastic optimization problem
\[
\min_{x \in \R^d} f(x) := \E_{\xi \sim \PP}\bigl[F(x;\xi)\bigr],
\]
where \((\Xi,\mathcal{F},\PP)\) is a probability space. For each \(\xi\), \(x\mapsto F(x,\xi)\) is differentiable, and for each \(x\), it is measurable with respect to \(\xi\). We assume that \(f\) is finite-valued and bounded below. The finite-sum case is recovered by taking \(\Xi=\{1,\dots,n\}\) and \(\PP\) to be the uniform distribution. 

In this paper, we consider the zeroth-order (ZO) setting where the algorithm can query only noisy function values \(F(x;\xi)\), with no
access to gradients. This oracle model arises naturally in black-box adversarial
attacks~\citep{chen2017zoo,ilyas2018black}, simulation-based
optimization~\citep{fu2015handbook}, and, more recently, memory-efficient
fine-tuning of large language models, where backpropagation is prohibitive and
ZO methods such as MeZO achieve competitive performance using only forward
passes~\citep{malladi2023mezo}.

For smooth nonconvex objectives, stochastic ZO methods
\citep{ghadimi2013stochastic,chayti2025random} establish the benchmark
\(O(d\,\varepsilon^{-4})\) query complexity for finding an
\(\varepsilon\)-stationary point in expectation, when each sample function \(x \mapsto F(x;\xi)\) is smooth and the gradient noise has bounded variance.

The bounded-variance assumption is, however, increasingly at odds with the
regimes in which ZO methods are deployed. Empirical evidence in modern machine
learning indicates that stochastic gradient noise is often \emph{heavy-tailed}.
For example, \citet{zhang2020adaptive} provide evidence of heavy-tailed
stochastic-gradient noise in BERT pretraining, with empirical variance estimates
that fail to stabilize as the sample size increases. Similar behavior has been
reported in image classification: \citet{simsekli2019tail} estimate
stochastic-gradient tail indices below \(2\), consistent with infinite-variance
noise. Heavy-tailed gradients have also been documented in reinforcement
learning~\citep{garg2021proximal}. In such regimes, the second moment may be
infinite or poorly controlled. This poses a challenge not only for first-order
gradient-based stochastic methods, but also for ZO methods, whose updates rely
on gradient estimates reconstructed from noisy function evaluations.

The first-order literature has responded with a substantial line of work showing
that \emph{gradient clipping}, possibly combined with momentum or normalization,
can restore convergence guarantees under bounded \(p\)-th moment noise on the
sample gradients for \(p \in (1,2]\)~\citep{gorbunov2020stochastic,
zhang2020adaptive,cutkosky2021high,sadiev2023high,nguyen2023improved,
liu2023breaking,chayti2026ransomsecondordermomentumrandomized}. For smooth nonconvex first-order stochastic optimization under
bounded \(p\)-th moment noise on the sample gradients, the standard benchmark is
the optimal in-expectation \(\varepsilon\)-dependence
\(\Theta(\varepsilon^{-(3p-2)/(p-1)})\)~\citet{zhang2020adaptive}. Matching high-probability guarantees
were later obtained by \citet{cutkosky2021high}.

\paragraph{The zeroth-order picture is much less developed.}
For zeroth-order optimization, results under heavy-tailed noise are scarce and
almost entirely confined to the convex setting. \citet{kornilov2024median}
obtain complexity bounds for nonsmooth convex ZO under symmetric heavy-tailed
noise using mini-batched median estimates of two-point gradient differences followed by vector clipping;
\citet{kornilov2023gradient} treat convex compact domains; and
\citet{kornilov2023accelerated} give an accelerated method when the variance is
infinite. To the best of our knowledge, no prior result appears to give a high-probability complexity bound for smooth
nonconvex stochastic ZO that matches the classical finite-variance expectation
benchmark. This work closes this gap under a \textit{weak-}$L_p$
tail condition on the sample gradient noise
$\nabla F(x;\xi)-\nabla f(x)$
(Assumption~\ref{ass:hp-shared-prob}), which is strictly weaker than the classical bounded $p$-th moment assumption and allows for infinite variance even when $p=2$.
\paragraph{Why the nonconvex ZO case is not a corollary.}
A natural temptation is to view ZO under heavy tails as an immediate consequence
of first-order results. In first-order SGD, the noise assumption is imposed
directly on the stochastic gradient noise
\(\nabla F(x;\xi)-\nabla f(x)\), and the update uses \(\nabla F(x;\xi)\) itself.
Thus the stochastic error appearing in the update is exactly the quantity
controlled by the assumption. In zeroth-order optimization, however, the
algorithm does not observe \(\nabla F(x;\xi)\). It only observes noisy function
values and forms directional finite-difference estimates. Consequently, the
error in the update is not the original sample-gradient noise vector, but a
scalar directional derivative multiplied by a random direction, together with
finite-difference and smoothing effects. Therefore, weak-\(L_p\) control of
\(\nabla F(x;\xi)-\nabla f(x)\) must first be transferred to the scalar
two-point estimator. This directional transfer is the main ZO-specific step in
our proof and is absent from standard first-order analyses. The difficulty is
amplified by the fact that our assumption is a weak-\(L_p\) tail condition,
which is strictly weaker than the classical bounded \(p\)-th moment assumption.
\paragraph{Contributions.}
We study a robust two-point ZO method, \textbf{RSC-ZO}, that combines
spherical smoothing, shared-randomness directional estimates, and
\emph{scalar clipping applied to each directional derivative before aggregation}
(Algorithm~\ref{alg:rsc-zo-unified}). Our contributions are as follows.

\begin{itemize}[leftmargin=1.2em,itemsep=2pt,topsep=2pt]

\item \textbf{High-probability ZO stationarity under weak-$L_p$ noise.}
Under sample-wise smoothness (Assumption~\ref{ass:sample-smooth}) and a
weak-$L_p$ tail assumption on the sample-gradient noise
(Assumption~\ref{ass:hp-shared-prob}), we prove that RSC-ZO finds an
$\varepsilon$-stationary point with high probability using
$
\widetilde{O}\!\left(
\frac{d^{\frac{p}{2(p-1)}}}
{\varepsilon^{\frac{3p-2}{p-1}}}
\right)
$
noisy function evaluations, up to logarithmic and problem-dependent factors
(Theorem~\ref{thm:hp-shared-clipped} and
Remark~\ref{rem:hp-shared-complexity}). At $p=2$, this gives
$\widetilde{O}(d\,\varepsilon^{-4})$, matching the benchmark
dimension--accuracy dependence of stochastic ZO methods under bounded-variance
noise~\citep{ghadimi2013stochastic,chayti2025random}. Unlike these classical
results, which are typically established in expectation, our guarantee holds
with high probability and under a weaker weak-$L_2$ tail condition. Moreover,
the $\varepsilon$-dependence matches the optimal first-order rate of
\citet{zhang2020adaptive}.

\item \textbf{A scalar-clipping analysis tailored to derivative-free estimation.}
A key step in the proof is to transfer weak-$L_p$ control of the sample-gradient
noise to clipped two-point directional estimates. This step is specific to the
zeroth-order setting: the algorithm never observes $\nabla F(x;\xi)$ directly,
but only noisy function values. The resulting analysis separates smoothing bias,
directional-estimation error, and heavy-tailed stochastic fluctuations before
aggregation.

\item \textbf{Momentum extension with a quantified batch-size/stability tradeoff.}
We extend the analysis to a momentum variant
(Algorithm~\ref{alg:rsc-zo-unified}). Using a localized Lyapunov
argument, we show that momentum reduces the required running batch size by a
factor of $1-\beta$ after warm start, at the price of the stability coupling
$\alpha \lesssim (1-\beta)/L$. The total query complexity is preserved (Corollary~\ref{cor:momentum-complexity}). Choosing $\beta$ close to $1$, we show that the same complexity holds with a batch size as small as $1$.

\end{itemize}

\paragraph{Comparison with related work.}
Our setting sits at the intersection of heavy-tailed first-order optimization and classical stochastic ZO theory.

\emph{Heavy-tailed first-order optimization.}
\citet{zhang2020adaptive} analyze clipped SGD under bounded $p$-th moment
sample-gradient noise and achieve the optimal in-expectation
$\varepsilon$-dependence for smooth nonconvex first-order stochastic
optimization. \citet{cutkosky2021high} obtain high-probability bounds using
clipping, momentum, and normalization. More recently,
\citet{hubler2024parameter,liu2024nonconvex} show that normalization alone,
without clipping, can achieve the optimal rate in related nonconvex first-order
settings. These results do not address the zeroth-order oracle, where the update
is constructed from noisy function values rather than observed stochastic
gradients.

\emph{Heavy-tailed zeroth-order optimization.}
The closest line of work is~\citet{kornilov2024median,kornilov2023gradient,
kornilov2023accelerated}, which studies \emph{convex}, typically nonsmooth,
zeroth-order optimization under heavy-tailed oracle noise. These works often
rely on additional structure, such as symmetry of the oracle noise, and use
median-based gradient-difference estimators followed by vector clipping. In
contrast, we treat \emph{smooth nonconvex} objectives, do not require symmetry,
and use scalar clipping directly on each directional derivative before
aggregation.

\emph{Stochastic ZO in the bounded-variance regime.}
\citet{ghadimi2013stochastic} and \citet{chayti2025random} establish the
benchmark $O(d\,\varepsilon^{-4})$ query complexity for stochastic ZO methods
under sample smoothness and bounded-variance noise, with guarantees
stated in expectation. 

\subsection{Assumptions}
\label{sec:setup}

We now state the assumptions used throughout the paper. Our analysis relies on two ingredients: sample-wise smoothness of the stochastic objective and a direct probability control on the sample gradient noise.

\begin{assumption}[Well-defined problem]\label{ass:Well-defined}
    The function $f$ is bounded from below, i.e.: $f_* := \inf_{x\in\mathbb{R}^d} f(x) > - \infty$ and we define the initial gap $\Delta_0 := f(x_0) - f_*$.
\end{assumption}

\begin{assumption}[Sample smoothness]
\label{ass:sample-smooth}
For $\mathbb P$-almost every $\xi$, the map $x\mapsto F(x;\xi)$ is $L$-smooth, i.e.,
$
\forall x,y\in\R^d,\quad
\norm{\nabla F(x;\xi)-\nabla F(y;\xi)}_2
\le
L\norm{x-y}_2.$
\end{assumption}

\begin{assumption}[Weak-$L_p$ control of the sample gradient noise]
\label{ass:hp-shared-prob}
There exist $p\in(1,2]$ and $\sigma>0$ such that, for every $x\in\R^d$ and every $t>0$,
$
\Pb_\xi\bigl(\norm{\nabla F(x;\xi)-\nabla f(x)}_2>t\bigr)
\le
\left(\frac{\sigma}{t}\right)^p.$
\end{assumption}

\begin{remark}
Assumption~\ref{ass:hp-shared-prob} is weaker than the bounded \(p\)-th moment condition used in previous works. Indeed,
a uniform bound
\[
\E_\xi \norm{\nabla F(x;\xi)-\nabla f(x)}_2^p \le \sigma^p
\]
implies Assumption~\ref{ass:hp-shared-prob} by Markov's inequality, while the converse need not
hold: weak-\(L_p\) control may allow an infinite \(p\)-th moment
(see Proposition~\ref{prop:p2-infinite-variance}).
\end{remark}
\begin{remark}
\label{rem:f-smooth}
Under Assumption~\ref{ass:sample-smooth}, in order to justify $\nabla f(x)=\E_\xi[\nabla F(x;\xi)]$ and conclude that $f$ is $L$-smooth, it is enough to assume that there exists $x_0\in\R^d$ such that $\E_\xi\norm{\nabla F(x_0;\xi)}_2<\infty$. Indeed, for any compact set $K\subset\R^d$,
$
\sup_{x\in K}\norm{\nabla F(x;\xi)}_2
\le
\norm{\nabla F(x_0;\xi)}_2 + L\sup_{x\in K}\norm{x-x_0}_2,$
so dominated convergence yields $\nabla f(x)=\E_\xi[\nabla F(x;\xi)]$. Then
\[
\norm{\nabla f(x)-\nabla f(y)}_2
\le
\E_\xi\norm{\nabla F(x;\xi)-\nabla F(y;\xi)}_2
\le
L\norm{x-y}_2,
\]
which shows that $f$ is $L$-smooth. Throughout the paper, we assume that such a point $x_0$ exists.
\end{remark}

\subsection{Algorithm}
\label{sec:robust-directional-clipping}

We now introduce RSC-ZO and its momentum variant. Both methods use the same
scalar-clipped two-point estimator. Given a point $x\in\mathbb R^d$, batch size
$M$, and clipping threshold $\tau$, sample
$u_\ell\sim\mathrm{Unif}(\mathbb S^{d-1})$ and $\xi_\ell\sim\mathbb P$, using
the same $\xi_\ell$ in the two function evaluations, and define
\[
Y_\ell(x):=
\frac{F(x+\mu u_\ell;\xi_\ell)-F(x-\mu u_\ell;\xi_\ell)}{2\mu}.
\]
We clip each scalar directional estimate by
$\psi_\tau(z):=z\min\{1,\tau/|z|\}$ and form the clipped zeroth-order estimator
$
G_\mu(x;M,\tau):=
\frac{d}{M}\sum_{\ell=1}^M \psi_\tau(Y_\ell(x))u_\ell.$ The base method corresponds to $\beta=0$ and updates directly with
$G_\mu(x_t;M,\tau)$. The momentum variant corresponds to $\beta>0$: in this
case, we first initialize the momentum variable using a separate warm-start
estimate at $x_0$, with batch size $M_0$ and threshold $\tau_0$, and then run
the momentum recursion. The warm start is used only in the momentum case.
Algorithm~\ref{alg:rsc-zo-unified} summarizes both variants.

\begin{algorithm}[H]
\caption{Robust Scalar-Clipped Zeroth-Order Method (RSC-ZO)}
\label{alg:rsc-zo-unified}
\begin{algorithmic}[1]
\Require $x_0\in\mathbb R^d$, stepsize $\alpha>0$, smoothing parameter $\mu>0$,
running batch size $M$, threshold $\tau$, momentum parameter $\beta\in[0,1)$

\If{$\beta=0$}
    \Statex \textbf{Base RSC-ZO}
    \For{$t=0,1,2,\dots$}
        \State $g_t\gets G_\mu(x_t;M,\tau)$
        \State $x_{t+1}\gets x_t-\alpha g_t$
    \EndFor
\Else
    \Statex \textbf{Momentum RSC-ZO}
    \Statex \textbf{Additional inputs:} warm-start batch size $M_0$, warm-start threshold $\tau_0$
    \State $g_0\gets G_\mu(x_0;M_0,\tau_0)$
    \State $m_0\gets g_0$
    \State $x_1\gets x_0-\alpha m_0$
    \For{$t=1,2,\dots$}
        \State $g_t\gets G_\mu(x_t;M,\tau)$
        \State $m_t\gets \beta m_{t-1}+(1-\beta)g_t$
        \State $x_{t+1}\gets x_t-\alpha m_t$
    \EndFor
\EndIf
\end{algorithmic}
\end{algorithm}
\subsection{Spherical smoothing and notation}

For $\mu>0$, define the spherical smoothing of $f$ by
$
f_\mu(x):=\E_{v\sim \Unif(\mathbb B^d)}[f(x+\mu v)],$
where $\mathbb B^d:=\{v\in\mathbb R^d:\norm{v}_2\le 1\}$ is the unit Euclidean ball. It is a classical result \citep{flaxman2004online} that
\[
\nabla f_\mu(x)=\frac{d}{\mu}\,\E_{u\sim \Unif(\mathbb S^{d-1})}[f(x+\mu u)u].
\]
Using the symmetry of the uniform distribution on $\mathbb S^{d-1}$, this can be rewritten in the two-point form
\[
\nabla f_\mu(x)=d\,\E_{u\sim \Unif(\mathbb S^{d-1})}\!\left[\frac{f(x+\mu u)-f(x-\mu u)}{2\mu}\,u\right].
\]
Accordingly, for $x\in\mathbb R^d$ and $u\in\mathbb S^{d-1}$, we define
$
D_\mu f(x,u):=\frac{f(x+\mu u)-f(x-\mu u)}{2\mu},$
so that
\[
\nabla f_\mu(x)=d\,\E_{u\sim \Unif(\mathbb S^{d-1})}[D_\mu f(x,u)u].
\]
\section{High-Probability Analysis of Base RSC-ZO}

\paragraph{Notation.}
Let
$
\bar\Delta_0 := \Delta_0 + \frac{L\mu^2}{2}.$
For brevity, define
$
S_\mu
:=
\frac{\sqrt{L\bar\Delta_0}+\sigma}{\sqrt d}
+
L\mu .$
This scale collects the three quantities that control the scalar directional
finite-difference estimates. The term
\(\sqrt{L\bar\Delta_0}/\sqrt d\) controls the noiseless directional derivative,
the term \(\sigma/\sqrt d\) controls the stochastic gradient noise after
projection onto the random direction \(u\), and the term \(L\mu\) controls the
finite-difference smoothing error. After multiplication by the zeroth-order
factor \(d\), the corresponding estimator scale is \(dS_\mu\).

\paragraph{Goal.}
We first establish a high-probability bound on the estimator error
\(g(x)-\nabla f_\mu(x)\). For each sample, define the two-point directional
derivative estimate
\[
Y_\ell :=
\frac{F(x+\mu u_\ell;\xi_\ell)-F(x-\mu u_\ell;\xi_\ell)}{2\mu}.
\]
The scalar-clipped zeroth-order estimator is then
\[
g(x) :=
\frac{d}{M}\sum_{\ell=1}^M \psi_\tau(Y_\ell)\,u_\ell,
\]
where the pairs \((u_\ell,\xi_\ell)\) are i.i.d., with
\(u_\ell\sim \mathrm{Unif}(\mathbb S^{d-1})\) and \(\xi_\ell\sim\mathcal P\).

We decompose the estimator error as
\[
g(x)-\nabla f_\mu(x)
=
\underbrace{g(x)-\mathbb E[g(x)]}_{A:\ \text{fluctuation}}
+
\underbrace{\mathbb E[g(x)]-\nabla f_\mu(x)}_{B:\ \text{clipping bias}}.
\]
The term \(A\) captures the empirical fluctuation of the clipped average, while
\(B\) captures the deterministic bias introduced by clipping. We control these
terms separately in Lemma~\ref{lem:hp-shared-empirical-fluctuation} and
Lemma~\ref{lem:hp-shared-clipping-bias}. Choosing the clipping level to balance
the two terms, up to logarithmic factors, yields the following localized
deviation bound for a single batched estimator at a fixed, non-random point
\(x\). The extension from this fixed-point estimate to the random trajectory is
handled by localization and a union bound in the proof of
Theorem~\ref{thm:main}.

\begin{theorem}[Localized deviation with tuned clipping]
\label{thm:tuned-dev}
Assume that Assumptions~\ref{ass:Well-defined}, \ref{ass:sample-smooth}, and
\ref{ass:hp-shared-prob} hold. Fix $\delta \in (0,1]$ and $T \ge 3$, set
$
\lambda := \log(T/\delta),$
and assume $M \ge \lambda$. Let $x \in \mathbb{R}^d$ satisfy
$
f(x)-f_* \le 4\bar\Delta_0,$
and choose
$
\tau := 8\,S_\mu\left(\frac{M}{\lambda}\right)^{1/p}.$
Then, with probability at least $1-\delta/T$,
\[
\|g(x)-\nabla f_\mu(x)\|_2
\le
C_p\,d\,S_\mu
\left(\frac{\lambda}{M}\right)^{\frac{p-1}{p}}
\sqrt{1+\log(M/\lambda)},
\]
where
$
C_p := 120+\frac{2^{4-p}}{p-1}.
$
\end{theorem}

All intermediate lemmas appear in Appendix~\ref{sec:hp-clipped-shared}. These
include the clipping-bias bound under weak-$L_p$ control, directional
weak-$L_p$ bounds for the noiseless signal and the shared-randomness estimator
$Y$, the clipped second-moment estimate, and the vector Bernstein inequality.
The corresponding detailed appendix statement is
Corollary~\ref{cor:hp-shared-localized-tuned}.

\medskip
We turn the localized deviation bound of Theorem~\ref{thm:tuned-dev} into
a descent estimate for the smoothed objective $f_\mu$. Since the update uses
the clipped estimator $g(x_t)$ in place of $\nabla f_\mu(x_t)$, the key step
is to combine the smoothness inequality for $f_\mu$ with the high-probability
control of $g(x_t)-\nabla f_\mu(x_t)$.

\begin{proposition}[One-step descent inequality]
\label{prop:descent-main}
Assume $f$ is $L$-smooth. For $x\in\mathbb R^d$, let
$
x^+ := x-\alpha g(x)$ and $
e(x):=g(x)-\nabla f_\mu(x).$
If $\alpha\le \frac{1}{4L}$, then
\[
f_\mu(x^+)
\le
f_\mu(x)
-
\frac{\alpha}{2}\|\nabla f_\mu(x)\|_2^2
+
\frac{5\alpha}{4}\|e(x)\|_2^2 .
\]
\end{proposition}

For $M\in\mathbb N$, $\mu>0$, $\delta\in(0,1)$, and $T\ge 3$, define
\[
\eta_0(M,\mu,\delta,T)
:=
C_p\,d\,S_\mu
\left(\frac{\log(T/\delta)}{M}\right)^{\frac{p-1}{p}}
\sqrt{
1+\log\!\left(\frac{M}{\log(T/\delta)}\right)
},
\]
and
$
\eta(M,\mu,\delta,T)
:=
\eta_0(M,\mu,\delta,T)+L\mu .$
Iterating Proposition~\ref{prop:descent-main} along the trajectory of
Algorithm~\ref{alg:rsc-zo-unified}, applying the localized deviation bound of
Theorem~\ref{thm:tuned-dev} on the event that the iterates remain in the region
\(f(x_t)-f_* \le 4\bar\Delta_0\), and taking a union bound over $T$ steps yields
our main result. The localization argument is given in
Appendix~\ref{sec:hp-clipped-shared}.

\begin{theorem}[High-probability convergence of base RSC-ZO]
\label{thm:main}
Assume that Assumptions~\ref{ass:Well-defined}, \ref{ass:sample-smooth}, and
\ref{ass:hp-shared-prob} hold. Let $\mu>0$,
$
\Delta_0 := f(x_0)-f_*,$ and $
\bar\Delta_0 := \Delta_0+\frac{L\mu^2}{2}.$
Let $\epsilon>0$, $\delta\in(0,1]$, $M\in\mathbb N$, and define
\[
T
:=
\max\left\{
3,\,
\left\lceil \frac{32L\bar\Delta_0}{\epsilon^2}\right\rceil
\right\}.
\]
Run Algorithm~\ref{alg:rsc-zo-unified} with $\beta=0$ (no momentum), batch size $M$, stepsize
$
\alpha := \frac{1}{4L},$
and clipping threshold
$
\tau
:=
8\,S_\mu
\left(\frac{M}{\log(T/\delta)}\right)^{1/p}.$
Assume
\[
M\ge \log(T/\delta),
\qquad
\eta(M,\mu,\delta,T)\le \frac{\epsilon}{4},
\qquad
\epsilon^2\le 32L\bar\Delta_0 .
\]
Then, with probability at least $1-\delta$,
\[
\frac1T\sum_{t=0}^{T-1}
\|\nabla f(x_t)\|_2^2
\le
\epsilon^2 .
\]
\end{theorem}

\begin{remark}[Query complexity]
\label{rem:complexity}
The total number of stochastic function queries is $2MT$. Choose
$\mu\asymp \epsilon/(Ld)$. Then $dL\mu=O(\epsilon)$, $L\mu=O(\epsilon/d)$, and
$\bar\Delta_0=\Delta_0+O(\epsilon^2/(Ld^2))$. Since
$
T=\max\left\{3,\left\lceil \frac{32L\bar\Delta_0}{\epsilon^2}\right\rceil\right\},$
we have $T=O(1+L\Delta_0/\epsilon^2)$. When the second term dominates the
constant \(3\), this is $T=O(L\Delta_0/\epsilon^2)$.

Up to logarithmic factors, the condition
$\eta(M,\mu,\delta,T)\le \epsilon/4$ reduces to
$
dS_\mu\left(\frac{\log(T/\delta)}{M}\right)^{\frac{p-1}{p}}
\lesssim \epsilon .$
Moreover, under the above choice of $\mu$,
\[
dS_\mu
=
\sqrt d\,(\sqrt{L\bar\Delta_0}+\sigma)+dL\mu
=
\sqrt d\,(\sqrt{L\Delta_0}+\sigma)+O(\epsilon).
\]
Therefore it is enough to choose
$
M
\gtrsim
\log(T/\delta)
\left(
\frac{\sqrt d\,(\sqrt{L\Delta_0}+\sigma)}{\epsilon}
\right)^{\frac{p}{p-1}},$
up to logarithmic factors. Combining this with $T=O(L\Delta_0/\epsilon^2)$ gives
$
MT
=
\widetilde O\!\left(
L\Delta_0
\bigl(\sqrt{L\Delta_0}+\sigma\bigr)^{\frac{p}{p-1}}
\frac{
d^{\frac{p}{2(p-1)}}
}{
\epsilon^{\frac{3p-2}{p-1}}
}
\right).$
Thus the total number of stochastic function queries, $2MT$, has the same order
up to the leading factor $2$. In particular, at $p=2$,
$
MT
=
\widetilde O\!\left(
L\Delta_0
(\sqrt{L\Delta_0}+\sigma)^2
\,d\,\epsilon^{-4}
\right),$
so the method has the classical $d\,\epsilon^{-4}$ dimension--accuracy
dependence up to logarithmic and problem-dependent factors. This matches the
classical bounded-variance ZO benchmark
\citep{ghadimi2013stochastic,chayti2025random}, but under the strictly weaker
weak-$L_2$ tail condition and as a high-probability, rather than
in-expectation, guarantee. The $\epsilon$-dependence further matches the
optimal first-order rate of \citet{zhang2020adaptive}.
\end{remark}

\begin{remark}[Trivial regime]
\label{rem:trivial}
The condition $\epsilon^2\le 32L\bar\Delta_0$ isolates the regime where the
iteration bound above is needed. If $\epsilon^2>32L\bar\Delta_0$, then
Lemma~\ref{lem:hp-shared-grad-subopt} gives
\[
\|\nabla f(x_0)\|_2^2
\le
2L(f(x_0)-f_*)
=
2L\Delta_0
\le
2L\bar\Delta_0
<
\epsilon^2 .
\]
Hence the desired stationarity conclusion already holds at the initial point.
\end{remark}
\section{High-Probability Analysis of Momentum RSC-ZO}
\label{sec:momentum-main}

We now study the momentum variant of RSC-ZO, corresponding to \(\beta>0\) in
Algorithm~\ref{alg:rsc-zo-unified}. In this variant, the algorithm averages
clipped directional estimates over time, which reduces the running batch size
required per iteration (see~\citep[Section 2.2]{chayti2025improvingstochasticcubicnewton} for a discussion on how momentum reuses past iterates to mimic a large batch). This reduction comes with a proportional stepsize
restriction of order \((1-\beta)/L\).

The method first computes a warm-start estimator
$
  g_0 = G_\mu(x_0;M_0,\tau_0),\,
  m_0=g_0$ and $
  x_1=x_0-\alpha m_0.$
Then, for \(t\ge1\), it forms an independent clipped two-point estimator
$
  g_t
  =
  G_\mu(x_t;M,\tau)
  =
  \frac dM
  \sum_{\ell=1}^M
  \psi_\tau(Y_{t,\ell})u_{t,\ell},$
and updates
\[
  m_t=\beta m_{t-1}+(1-\beta)g_t,
  \qquad
  x_{t+1}=x_t-\alpha m_t,
\]
where \(\beta\in(0,1)\) is the momentum parameter. The analysis below focuses
on the high-momentum regime \(\beta\in[1/2,1)\), where the batch-size reduction
is most relevant.

\begin{theorem}[High-probability convergence of momentum RSC-ZO]
\label{thm:momentum-main}
Suppose Assumptions~\ref{ass:Well-defined}, \ref{ass:sample-smooth}, and
\ref{ass:hp-shared-prob} hold. Let \(\epsilon\in(0,1)\),
\(\delta\in(0,1]\), and \(\beta\in[1/2,1)\). Choose
$
  \mu \le \frac{\epsilon}{4Ld}.$
Let
$
  \Delta_0:=f(x_0)-f_*$ and $
  \bar\Delta_0:=\Delta_0+\frac{L\mu^2}{2}$
and assume \(\epsilon^2\le 32L\bar\Delta_0\). Define
$
  S_\mu
  :=
  \frac{\sqrt{L\bar\Delta_0}+\sigma}{\sqrt d}+L\mu$ and $
  \mathcal R_\mu:=dS_\mu.$
Choose
\[
  \alpha=\frac{1-\beta}{16\sqrt 3\,L},
  \qquad
  T
  =
  \left\lceil
    512\sqrt 3\,
    \frac{L\bar\Delta_0}{(1-\beta)\epsilon^2}
  \right\rceil+2,
\]
and set
\[
  \lambda:=\log\frac{2T}{\delta},
  \qquad
  \lambda_0:=\log\frac{2}{\delta}.
\]
Let
\[
  \tau
  =
  8S_\mu
  \left(
    \frac{M}{(1-\beta)\lambda}
  \right)^{1/p},
  \qquad
  \tau_0
  =
  8S_\mu
  \left(
    \frac{M_0}{\lambda_0}
  \right)^{1/p}.
\]
Let \(M_0,M\in\mathbb N\). Assume \(M_0\ge \lambda_0\),
\(M\ge (1-\beta)\lambda\), and
\[
  M_0
  =
  \widetilde{\Theta}
  \left(
    \left(
      \frac{\mathcal R_\mu}{\epsilon}
    \right)^{\frac{p}{p-1}}
  \right),
  \qquad
  M
  =
  \widetilde{\Theta}
  \left(
    (1-\beta)
    \left(
      \frac{\mathcal R_\mu}{\epsilon}
    \right)^{\frac{p}{p-1}}
  \right).
\]
Then, with probability at least \(1-\delta\),  we have   
$
  \frac1{T-2}
  \sum_{t=1}^{T-2}
  \|\nabla f(x_t)\|_2^2
  \le
  \epsilon^2.$
\end{theorem}

\begin{remark}[Query complexity and momentum tradeoff]
\label{rem:momentum-complexity}
The notation \(\widetilde{\Theta}(\cdot)\) hides only logarithmic factors and
constants depending on \(p\), but not on
\(d,L,\Delta_0,\sigma,\epsilon,\delta\), or \(\beta\). The total number of
stochastic function evaluations is
$
  Q_{\rm MOM}:=2M_0+2M(T-1).$
Using the choices in Theorem~\ref{thm:momentum-main}, we obtain
$
  Q_{\rm MOM}
  =
  \widetilde O
  \left(
    \left(
      \frac{\mathcal R_\mu}{\epsilon}
    \right)^{\frac{p}{p-1}}
  \right)
  +
  \widetilde O
  \left(
    \frac{L\bar\Delta_0}{\epsilon^2}
    \left(
      \frac{\mathcal R_\mu}{\epsilon}
    \right)^{\frac{p}{p-1}}
  \right).$
Moreover, under the above choice of \(\mu\),
$
  \bar\Delta_0
  =
  \Delta_0+O\!\left(\frac{\epsilon^2}{Ld^2}\right)$ and $
  \mathcal R_\mu
  =
  \sqrt d(\sqrt{L\Delta_0}+\sigma)+O(\epsilon).$
Hence
$
  Q_{\rm MOM}
  =
  \widetilde O
  \left(
    \left(
      \frac{\sqrt d(\sqrt{L\Delta_0}+\sigma)}{\epsilon}
    \right)^{\frac{p}{p-1}}
  \right)
  +
  \widetilde O
  \left(
    \frac{L\Delta_0}{\epsilon^2}
    \left(
      \frac{\sqrt d(\sqrt{L\Delta_0}+\sigma)}{\epsilon}
    \right)^{\frac{p}{p-1}}
  \right).$
Thus, compared with base RSC-ZO, momentum reduces the running batch size by a
factor of order \(1-\beta\), while the admissible stepsize scales as
$
  \alpha\asymp \frac{1-\beta}{L}.$
The per-iteration query cost decreases, but the number of iterations increases
by the reciprocal factor. The total query complexity is therefore preserved up
to logarithmic factors and the additive warm-start cost.
\end{remark}

\textbf{Momentum reduces the batch size requirement of Theorem~\ref{thm:main}}. In particular, it suffices to take $\beta= 1 - \widetilde{\Theta}\left(\big(\frac{\epsilon}{\mathcal R_\mu}
    \big)^{\frac{p}{p-1}}\right)$, and $T = \widetilde O
  \left(
    \frac{L\bar\Delta_0}{\epsilon^2}
    \left(
      \frac{\mathcal R_\mu}{\epsilon}
    \right)^{\frac{p}{p-1}}
  \right)$ to guarantee    
$
  \frac1{T-2}
  \sum_{t=1}^{T-2}
  \|\nabla f(x_t)\|_2^2
  \le
  \epsilon^2.$ with probability at least \(1-\delta\), and most importantly, this is obtained for a batch size $M$ constant that can be even $M=1$ (note that $M_0$ will stay unchanged).

\section{Zeroth-Order Fine-Tuning Experiments}
\label{sec:mezo-experiments}

\subsection{Setup}

We evaluate RSC-ZO on classification fine-tuning of RoBERTa-large
(350M parameters)~\citep{liu2019roberta}. We consider two natural language
inference tasks, SNLI~\citep{bowman2015snli} and
MNLI~\citep{williams2018mnli}, and one topic classification task,
TREC~\citep{voorhees2000trec}. We follow the prompt-based fine-tuning setup of
\citet{malladi2023mezo} for masked language models~\citep{gao2021making}, using
\(k=512\) training examples per task.

\subsection{Results}

We compare RSC-ZO against MeZO as the primary baseline. We also evaluate
multi-direction variants of RSC-ZO, where the zeroth-order estimate is averaged
over \(M\in\{1,2,4\}\) directions. Since increasing \(M\) increases the number
of forward passes per update, we keep the forward-evaluation budget fixed by
scaling the data batch size inversely with \(M\).

Table~\ref{tab:roberta} reports test accuracy over five seeds. RSC-ZO improves
substantially over MeZO on all three tasks. The gains are already strong for
\(M=1\), showing that scalar clipping is useful even in the standard one-direction
MeZO regime. Increasing \(M\) improves MNLI and TREC further, while SNLI performs
best at \(M=1\), suggesting that the best number of directions can be
task-dependent under a fixed query budget.

\begin{table}[H]
\centering
\begin{tabular}{lccc}
\toprule
\textbf{Method} & \textbf{SNLI} & \textbf{MNLI} & \textbf{TREC} \\
\midrule
Zero-shot & 50.4 & 49.2 & 32.0 \\
\midrule
MeZO & 74.2 (1.5) & 67.9 (2.8) & 83.2 (4.3) \\
\midrule
RSC-ZO (\(M=1\)) & \textbf{83.0 (1.0)} & 75.5 (0.9) & 92.4 (1.2) \\
RSC-ZO (\(M=2\)) & 82.0 (0.2) & 76.2 (1.5) & 92.9 (1.2) \\
RSC-ZO (\(M=4\)) & 79.7 (0.7) & \textbf{76.9 (1.3)} & \textbf{93.5 (0.6)} \\
\bottomrule
\end{tabular}
\vspace{6pt}
\caption{Accuracy of RoBERTa-large fine-tuned with \(k=512\) examples. Results
are mean accuracy with standard deviation over five seeds.}
\label{tab:roberta}
\end{table}

All methods are trained for 100K steps with learning rate \(10^{-6}\), smoothing
parameter \(\mu=10^{-3}\) (both optimized for MeZo), and no weight decay. For RSC-ZO, the clipping
threshold \(\tau\) is selected per task by grid search over
\(\{2,4,8,16,32,64\}\). Additional implementation details are given in
Appendix~\ref{app:mezo-details}.

\section{Controlled experiment: clipping before versus after aggregation}
\label{subsec:controlled-exp-main}

We include a controlled experiment to isolate the estimator-level effect of
scalar clipping. The goal is not to benchmark on a difficult objective, but to
test the mechanism suggested by our theory: in zeroth-order estimation under
heavy-tailed noise, it is beneficial to clip each scalar directional derivative
before aggregation, rather than only clipping the final batched vector.

We consider
$
    f(x)=\frac12\|x\|_2^2$ and $
    F(x;\zeta)=\frac12\|x\|_2^2+\langle \zeta,x\rangle$, 
where the noise is sparse and heavy-tailed,
$
    \zeta=sAe_J.$
Here \(s\) is uniform on \(\{-1,+1\}\), \(J\) is uniform on
\(\{1,\ldots,d\}\), and \(A\) is Pareto with tail exponent \(p\). Thus
\(\mathbb E[\zeta]=0\), the population objective is \(f\), and for
\(p\in(1,2)\) the stochastic gradient noise has finite first moment but infinite
variance. For a two-point shared-randomness estimator,
\[
    Y_\ell
    =
    \frac{F(x+\mu u_\ell;\zeta_\ell)-F(x-\mu u_\ell;\zeta_\ell)}{2\mu}
    =
    \langle x,u_\ell\rangle+\langle \zeta_\ell,u_\ell\rangle .
\]
We compare three natural estimators:
\[
\begin{aligned}
g_{\rm raw}
&=
\frac dM\sum_{\ell=1}^M Y_\ell u_\ell,
&&\text{unclipped ZO estimator},\\
g_{\rm vec}
&=
\operatorname{clip}_{r_{\rm vec}}(g_{\rm raw}),
&&\text{vector clipping after aggregation},\\
g_{\rm sc}
&=
\frac dM\sum_{\ell=1}^M \psi_\tau(Y_\ell)u_\ell,
&&\text{scalar clipping before aggregation}.
\end{aligned}
\]
The final-vector baseline is the direct analogue of standard gradient clipping
in SGD: one first forms a stochastic gradient-like vector and then clips its
norm. In contrast, scalar clipping acts earlier, at the level of each
one-dimensional directional measurement.

\begin{figure}[]
    \centering
    \begin{minipage}{0.48\linewidth}
        \centering
        \safeincludegraphics[width=\linewidth]{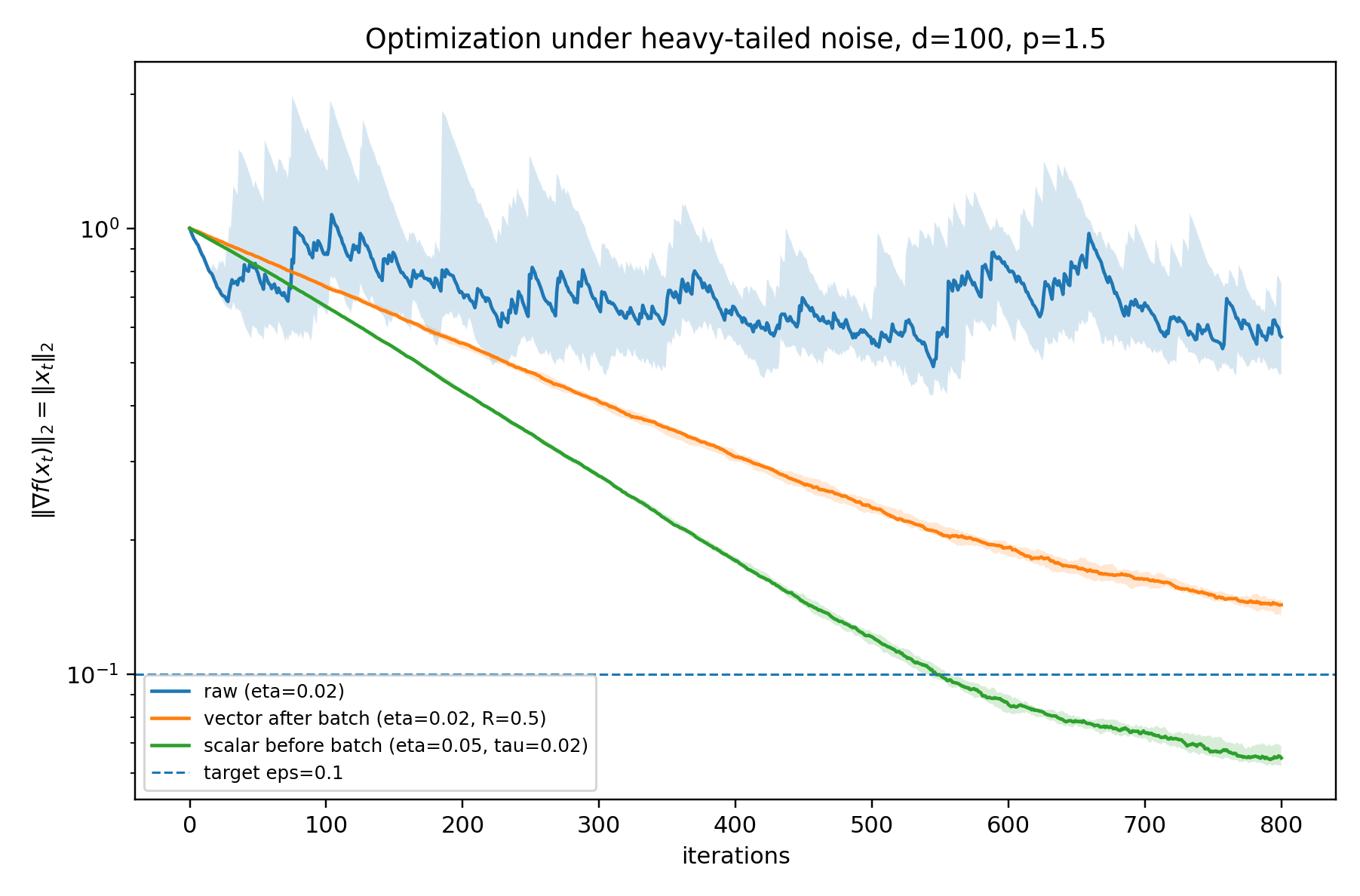}
    \end{minipage}
    \hfill
    \begin{minipage}{0.48\linewidth}
        \centering
        \safeincludegraphics[width=\linewidth]{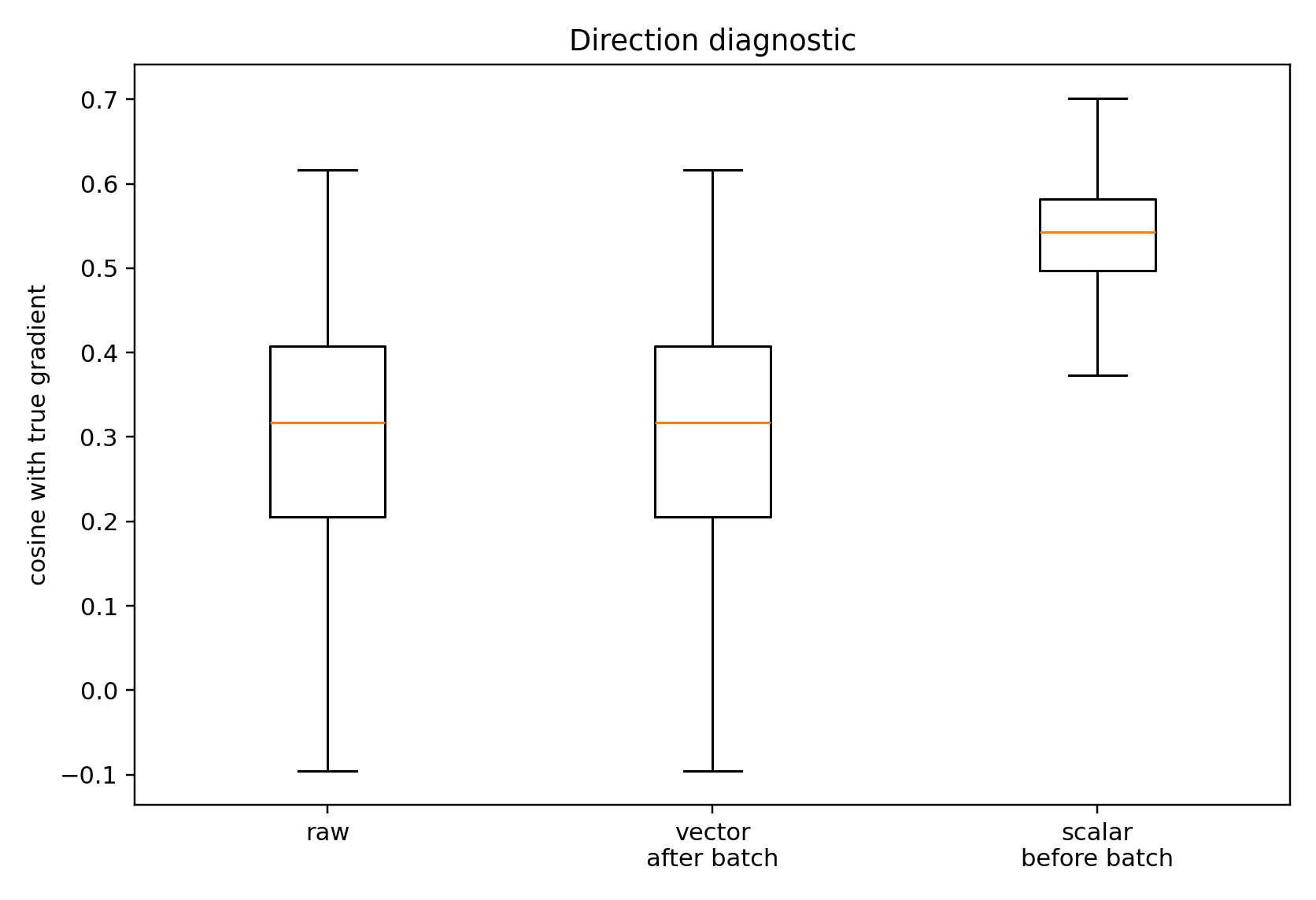}
    \end{minipage}
    \caption{
    Controlled heavy-tailed quadratic experiment with \(d=100\), Pareto tail
    exponent \(p=1.5\), and \(M=256\). Left: optimization curves for
    \(\|\nabla f(x_t)\|_2=\|x_t\|_2\). Raw ZO is unstable, final-vector clipping
    stabilizes but plateaus above the target, and scalar clipping before
    aggregation reaches the target. Right: cosine similarity with the true
    gradient. Final-vector clipping has essentially the same alignment as the
    raw estimator because it preserves the raw direction, whereas scalar
    clipping substantially improves alignment.
    }
    \label{fig:controlled-exp-main}
\end{figure}

Figure~\ref{fig:controlled-exp-main} shows that vector clipping after aggregation
improves over the raw estimator, but remains substantially worse than scalar
clipping before aggregation. The cosine diagnostic explains the gap:
\(\operatorname{clip}_{r_{\rm vec}}(g_{\rm raw})\) rescales the already-formed
estimator and therefore preserves its direction, while scalar clipping prevents
extreme directional derivatives from dominating the average. Additional
dimension and tail-index sweeps, fair hyperparameter tuning, and outlier
diagnostics are reported in Appendix~\ref{sec:experiments}.

\section{Discussion and Conclusion}

We studied stochastic zeroth-order optimization for smooth nonconvex objectives under heavy-tailed sample-gradient noise. The key challenge is that ZO methods do not observe the stochastic gradients on which the noise assumption is imposed; they only observe function values and construct directional finite-difference estimates. Our analysis shows that weak-$L_p$ control of the sample-gradient noise can be transferred to scalar two-point directional estimates under sample-wise smoothness and shared randomness.

We proposed RSC-ZO, which clips each scalar directional derivative before aggregation. This prevents extreme directional measurements from dominating the final gradient estimator. Under weak-$L_p$ noise, RSC-ZO achieves a high-probability complexity bound matching the optimal first-order $\epsilon$-dependence. At $p=2$, it recovers the classical $d\epsilon^{-4}$ stochastic ZO dependence, but under the weaker weak-$L_2$ condition and with high probability. We also analyzed a momentum variant, showing that momentum reduces the required running batch size while preserving the total query complexity up to logarithmic factors.

Our experiments support the proposed mechanism. In controlled heavy-tailed quadratic problems, scalar clipping improves the estimator direction compared with raw ZO estimation and final-vector clipping. In RoBERTa-large fine-tuning, RSC-ZO improves over MeZO under the same forward-pass budget.

\paragraph{Limitations.}
Our theory assumes sample-wise smoothness, weak-$L_p$ control of the sample-gradient noise, and shared randomness in the two-point estimator. These assumptions are natural for isolating the effect of heavy-tailed sample-gradient noise in two-point ZO methods, but they do not cover nonsmooth objectives, pure function-value noise models, or black-box systems where shared randomness is unavailable. The theoretically prescribed clipping threshold also depends on problem-specific quantities; in our experiments, we therefore tune the threshold by grid search. Developing adaptive clipping rules with comparable high-probability guarantees is an important direction for future work.

Our empirical results are intended to support the estimator-level mechanism predicted by the theory. The controlled heavy-tailed quadratic experiment isolates the benefit of clipping scalar directional derivatives before aggregation, while the RoBERTa-large fine-tuning experiments show that the same mechanism can improve MeZO-style training under a fixed forward-evaluation budget. A broader empirical study across model scales, tasks, hyperparameter budgets, and robust ZO alternatives would further clarify the practical scope of the method. Finally, although the dimension dependence matches classical stochastic ZO benchmarks at $p=2$, it remains a bottleneck in very high-dimensional problems. Combining scalar clipping with structured perturbations or low-dimensional subspace methods is a promising direction.

\bibliographystyle{plainnat}
\bibliography{references}

\newpage
\appendix

\section{Weak-\(L_2\) Noise Can Have Infinite Variance}

A natural question is whether the specialization \(p=2\) in
Assumption~\ref{ass:hp-shared-prob} simply recovers the classical bounded-variance setting. Indeed, by
Markov's inequality, any bounded-variance model satisfies Assumption~\ref{ass:hp-shared-prob} with
\(p=2\). More precisely, if
\[
\E_\xi\|\nabla F(x;\xi)-\nabla f(x)\|_2^2\le \sigma^2
\quad \text{for all } x\in\R^d,
\]
then
\[
\Pb\!\left(
\|\nabla F(x;\xi)-\nabla f(x)\|_2>t
\right)
\le
\frac{\sigma^2}{t^2}
\quad \text{for all } x\in \R^d,\ t>0.
\]
Thus, the bounded-variance regime is contained in our \(p=2\) framework. The converse, however, is
false: even at \(p=2\), Assumption~\ref{ass:hp-shared-prob} still allows stochastic models with infinite
gradient-noise variance. Therefore, the \(p=2\) case of our theory is not a reformulation of
bounded-variance results, but a genuine extension beyond them.

\begin{proposition}[The case \(p=2\) still allows infinite variance]
\label{prop:p2-infinite-variance}
Let \(f_0:\R^d\to\R\) be any \(L\)-smooth function, and let \(Z\) be a real-valued random variable
with density
\[
\rho(z)=|z|^{-3}\mathbf 1_{\{|z|\ge 1\}}.
\]
Define the stochastic objective
\[
F(x;\xi):=f_0(x)+Z(\xi)\,x_1,
\]
where \(x_1\) denotes the first coordinate of \(x\), and let
\[
f(x):=\E_\xi[F(x;\xi)].
\]
Then the following hold:
\begin{enumerate}
    \item \(f(x)=f_0(x)\), and for \(\mathbb P\)-almost every \(\xi\), the map \(x\mapsto F(x;\xi)\)
    is \(L\)-smooth.
    \item The sample gradient noise satisfies
    \[
    \nabla F(x;\xi)-\nabla f(x)=Z(\xi)e_1,
    \]
    where \(e_1=(1,0,\dots,0)^\top\).
    \item Assumption~\ref{ass:hp-shared-prob} holds with \(p=2\) and \(\sigma=1\), since for every \(t>0\),
    \[
    \Pb\!\left(
    \|\nabla F(x;\xi)-\nabla f(x)\|_2>t
    \right)
    =
    \Pb(|Z|>t)
    \le
    t^{-2}.
    \]
    \item The second moment is infinite:
    \[
    \E_\xi\|\nabla F(x;\xi)-\nabla f(x)\|_2^2
    =
    \E[Z^2]
    =
    \infty.
    \]
\end{enumerate}
\end{proposition}
\begin{proof}[Proof of Proposition~\ref{prop:p2-infinite-variance}]
Since \(\rho\) is symmetric, \(Z\) has zero mean whenever the first moment exists. Moreover,
\[
\E|Z|
=
\int_{\R}|z|\rho(z)\,dz
=
2\int_1^\infty z\cdot z^{-3}\,dz
=
2\int_1^\infty z^{-2}\,dz
<
\infty,
\]
so \(\E[Z]=0\). Therefore,
\[
f(x)=\E[F(x;\xi)]
=
f_0(x)+\E[Z]\,x_1
=
f_0(x).
\]
Also,
\[
\nabla F(x;\xi)=\nabla f_0(x)+Z(\xi)e_1,
\]
and hence
\[
\nabla F(x;\xi)-\nabla f(x)=Z(\xi)e_1.
\]
Because the added term \(Z(\xi)x_1\) is linear in \(x\), it does not affect smoothness; therefore
\(F(\cdot;\xi)\) is \(L\)-smooth whenever \(f_0\) is \(L\)-smooth.

Since
\[
\|\nabla F(x;\xi)-\nabla f(x)\|_2=|Z(\xi)|,
\]
for every \(t\ge 1\),
\[
\Pb(|Z|>t)
=
2\int_t^\infty z^{-3}\,dz
=
t^{-2},
\]
while for \(0<t<1\) one trivially has \(\Pb(|Z|>t)\le 1\le t^{-2}\). Thus
\[
\forall t>0,\qquad
\Pb\!\left(
\|\nabla F(x;\xi)-\nabla f(x)\|_2>t
\right)
\le
t^{-2},
\]
so Assumption~\ref{ass:hp-shared-prob} holds with \(p=2\) and \(\sigma=1\).

Finally,
\[
\E[Z^2]
=
\int_{\R} z^2 \rho(z)\,dz
=
2\int_1^\infty z^2\cdot z^{-3}\,dz
=
2\int_1^\infty z^{-1}\,dz
=
\infty.
\]
Hence the gradient noise has infinite variance.
\end{proof}

\paragraph{Takeaway.}
The case \(p=2\) is not simply a restatement of bounded-variance theory. It strictly enlarges the
noise model while preserving, up to logarithmic factors, the benchmark complexity
\[
O\!\left(d\,\sigma^2\,\epsilon^{-4}\right)
\]
known for stochastic zeroth-order methods in the \emph{individual smoothness} setting under a
bounded-variance assumption on the sample gradient noise
\citep{ghadimi2013stochastic,chayti2025random}. In particular,
\cref{thm:hp-shared-clipped} achieves the benchmark complexity bound
\(O(d\,\sigma^2\,\epsilon^{-4})\) up to logarithmic factors under the weaker polynomial-tail
assumption (see Remark~\ref{rem:hp-shared-complexity}).

\section{High-Probability Analysis of Base RSC-ZO}
\label{sec:hp-clipped-shared}

\paragraph{Notation.}
Define
$
f_*:=\inf_{x\in\mathbb R^d} f(x),
\Delta_0:=f(x_0)-f_*,
$
and
$
\bar\Delta_0:=\Delta_0+\frac{L\mu^2}{2}.
$
For brevity, write
\[
S_\mu
:=
\frac{\sqrt{L\bar\Delta_0}+\sigma}{\sqrt d}
+
L\mu .
\]

\paragraph{Goal of this section.}
Our goal is to prove a high-probability bound on the estimator error
\[
\|g(x)-\nabla f_\mu(x)\|_2,
\]
where
\[
Y_\ell
:=
\frac{F(x+\mu u_\ell;\xi_\ell)-F(x-\mu u_\ell;\xi_\ell)}{2\mu}
\quad\text{and}\quad
g(x)
:=
\frac{d}{M}\sum_{\ell=1}^M \psi_\tau(Y_\ell)u_\ell .
\]
We decompose the estimator error as
\[
g(x)-\nabla f_\mu(x)
=
\underbrace{(g(x)-\E[g(x)])}_{A}
+
\underbrace{(\E[g(x)]-\nabla f_\mu(x))}_{B}.
\]
Here, \(A\) is the empirical fluctuation of the clipped sample average, while
\(B\) is the bias introduced by clipping. Hence, to obtain a high-probability
upper bound on \(\|g(x)-\nabla f_\mu(x)\|_2\), it suffices to bound
\(\|A\|_2\) and \(\|B\|_2\) separately. We first consider \(B\).

Let
$
Y:=\frac{F(x+\mu u;\xi)-F(x-\mu u;\xi)}{2\mu},
$
where \(u\sim\Unif(\mathbb S^{d-1})\) is independent of \(\xi\). Then
\[
B
=
\E[g(x)]-\nabla f_\mu(x)
=
d\,\E\bigl[(\psi_\tau(Y)-Y)u\bigr].
\]
Therefore,
\[
\|B\|_2
=
d\,\Bigl\|\E\bigl[(\psi_\tau(Y)-Y)u\bigr]\Bigr\|_2
\le
d\,\E\bigl[|\psi_\tau(Y)-Y|\bigr].
\]

To control the quantity \(\E[|\psi_\tau(Y)-Y|]\), we first establish the following abstract clipping-bias bound.

\begin{lemma}[Clipping bias under weak-$L_p$ control]
\label{lem:hp-shared-clipped-bias-abstract}
Let \(Z\) be a real-valued random variable such that, for some \(S>0\),
\[
\forall t\ge 4S,
\qquad
\Pb(|Z|>t)\le 2^{2p+1}S^p t^{-p}.
\]
Then, for every \(\tau\ge 4S\),
\[
\E\bigl[|Z-\psi_\tau(Z)|\bigr]
\le
\frac{2^{2p+1}}{p-1}\,S^p\tau^{1-p}.
\]
\end{lemma}

It remains to verify that the directional estimator \(Y\) satisfies the hypothesis of
Lemma~\ref{lem:hp-shared-clipped-bias-abstract}. To this end, we use the decomposition
\[
Y=\bigl(Y-D_\mu f(x,u)\bigr)+D_\mu f(x,u).
\]
Thus, it is enough to control the tails of \(Y-D_\mu f(x,u)\) and
\(D_\mu f(x,u)\). The first term measures the deviation of the stochastic
directional estimator around its conditional mean, while the second is the
directional signal itself. These are controlled in
Lemmas~\ref{lem:hp-shared-directional-tail} and
\ref{lem:hp-shared-directional-signal-tail}, respectively.

\begin{lemma}[Directional weak-$L_p$ bound]
\label{lem:hp-shared-directional-tail}
Assume that Assumptions~\ref{ass:sample-smooth} and
\ref{ass:hp-shared-prob} hold. Let \(u\sim\Unif(\mathbb S^{d-1})\), independent
of \(\xi\). Then, for every fixed \(x\in\R^d\) and every \(s\ge 1\),
\[
\Pb_{u,\xi}\!\left(
\left|
\frac{F(x+\mu u;\xi)-F(x-\mu u;\xi)}{2\mu}
-
D_\mu f(x,u)
\right|
>
\left(\frac{\sigma}{\sqrt d}+L\mu\right)s
\right)
\le s^{-p}.
\]
\end{lemma}

\begin{lemma}[Localized weak-$L_p$ bound for the noiseless directional signal]
\label{lem:hp-shared-directional-signal-tail}
Assume that Assumptions~\ref{ass:Well-defined} and
\ref{ass:sample-smooth} hold. Let \(p\in(1,2]\). Let \(x\in\R^d\) satisfy
\(f(x)-f_*\le 4\bar\Delta_0\), and let
\(u\sim\Unif(\mathbb S^{d-1})\). Then, for every \(s\ge 1\),
\[
\Pb_u\!\left(
|D_\mu f(x,u)|
>
\left(\frac{\sqrt{8L\bar\Delta_0}}{\sqrt d}+L\mu\right)s
\right)
\le s^{-p}.
\]
\end{lemma}

\begin{lemma}[Localized weak-$L_p$ bound for the shared-randomness directional estimator]
\label{lem:hp-shared-directional-total-tail}
Assume that Assumptions~\ref{ass:Well-defined}, \ref{ass:sample-smooth}, and
\ref{ass:hp-shared-prob} hold. Let \(x\in\R^d\) satisfy
\(f(x)-f_*\le 4\bar\Delta_0\). Let
\(u\sim\Unif(\mathbb S^{d-1})\), independent of \(\xi\), and define
$
Y:=\frac{F(x+\mu u;\xi)-F(x-\mu u;\xi)}{2\mu}.
$
Then, for every \(s\ge 1\),
\[
\Pb\!\left(
|Y|>
4S_\mu s
\right)
\le 2s^{-p}.
\]
Equivalently, for every \(t\ge 4S_\mu\),
\[
\Pb(|Y|>t)
\le
2^{2p+1}
S_\mu^p t^{-p}.
\]
\end{lemma}

We may now combine the preceding ingredients to obtain the following deterministic bound on the clipping bias term \(B\).

\begin{lemma}[Bias induced by clipping]
\label{lem:hp-shared-clipping-bias}
Assume that Assumptions~\ref{ass:Well-defined}, \ref{ass:sample-smooth}, and
\ref{ass:hp-shared-prob} hold. Let \(x\in\R^d\) satisfy
$
f(x)-f_*\le 4\bar\Delta_0.
$
Let \(u\sim\Unif(\mathbb S^{d-1})\), independent of \(\xi\), and set
$
Y:=
\frac{F(x+\mu u;\xi)-F(x-\mu u;\xi)}{2\mu}.
$
Then, for every \(\tau\ge 4S_\mu\),
\[
\Bigl\|d\,\E[\psi_\tau(Y)u]-\nabla f_\mu(x)\Bigr\|_2
\le
\frac{2^{2p+1}}{p-1}\,
d\,S_\mu^p\,\tau^{1-p}.
\]
Equivalently, if
$
B:=\E[g(x)]-\nabla f_\mu(x),
$
then
\[
\|B\|_2
\le
\frac{2^{2p+1}}{p-1}\,
d\,S_\mu^p\,\tau^{1-p}.
\]
\end{lemma}

We now turn to the empirical fluctuation term
\[
A
=
\frac{d}{M}\sum_{\ell=1}^M
\Bigl(\psi_\tau(Y_\ell)u_\ell-\E[\psi_\tau(Y_\ell)u_\ell]\Bigr).
\]
The summands are i.i.d. centered random vectors. We therefore control \(A\)
using the vector Bernstein inequality of Lemma~\ref{lem:hp-shared-vector-bernstein}.
To apply that lemma, it remains to verify an almost-sure bound on each summand
and a suitable second-moment bound. For each \(\ell\), define
\[
X_\ell
:=
\frac{d}{M}
\Bigl(\psi_\tau(Y_\ell)u_\ell-\E[\psi_\tau(Y_\ell)u_\ell]\Bigr).
\]
Since \(\|u_\ell\|_2=1\) almost surely and \(|\psi_\tau(Y_\ell)|\le \tau\), we have
\[
\|X_\ell\|_2
\le
\frac{d}{M}\Bigl(|\psi_\tau(Y_\ell)|+\|\E[\psi_\tau(Y_\ell)u_\ell]\|_2\Bigr)
\le
\frac{2d\tau}{M}.
\]
This verifies the first condition in Lemma~\ref{lem:hp-shared-vector-bernstein}.

To verify the second condition, we use the following clipped second-moment estimate.

\begin{lemma}[Clipped second moment under weak-$L_p$ control]
\label{lem:hp-shared-clipped-second-moment}
Let \(Z\) be a real-valued random variable such that, for some \(S>0\),
\[
\forall t\ge 4S,
\quad
\Pb(|Z|>t)\le 2^{2p+1}S^p t^{-p}.
\]
Then, for every \(1<p\le 2\) and every \(\tau\ge 4S\),
\[
\E[\psi_\tau(Z)^2]
\le
64\,S^p\tau^{2-p}\bigl(1+\log(\tau/S)\bigr).
\]
\end{lemma}

We are now ready to combine the vector Bernstein inequality with the clipped
second-moment estimate to control the empirical fluctuation term
$
A=g(x)-\E[g(x)].
$
This yields the following high-probability bound.

\begin{lemma}[Empirical fluctuation of the clipped estimator]
\label{lem:hp-shared-empirical-fluctuation}
Assume that Assumptions~\ref{ass:Well-defined}, \ref{ass:sample-smooth}, and
\ref{ass:hp-shared-prob} hold. Let \(x\in\mathbb R^d\) satisfy
$
f(x)-f_*\le 4\bar\Delta_0.
$
Let \(\rho\in(0,1)\), and let \(\tau\ge 4S_\mu\). For \(\ell=1,\dots,M\), let
\[
Y_\ell:=
\frac{F(x+\mu u_\ell;\xi_\ell)-F(x-\mu u_\ell;\xi_\ell)}{2\mu},
\qquad
\widetilde Y_\ell:=\psi_\tau(Y_\ell),
\]
where the pairs \((u_\ell,\xi_\ell)\) are independent across \(\ell\), with
\(u_\ell\sim\Unif(\mathbb S^{d-1})\) and \(\xi_\ell\sim\mathbb P\). Define
\[
A:=
\frac{d}{M}\sum_{\ell=1}^M
\Bigl(
\widetilde Y_\ell u_\ell-\E[\widetilde Y_\ell u_\ell]
\Bigr).
\]
Then, with probability at least \(1-\rho\),
\[
\begin{aligned}
\|A\|_2
\le\;&
\frac{8d}{\sqrt M}
S_\mu^{p/2}
\tau^{1-\frac p2}
\sqrt{
1+
\log\left(
\frac{\tau}{S_\mu}
\right)
}
\left(1+\sqrt{2\log(1/\rho)}\right)
+
\frac{8d\tau}{3M}\log(1/\rho).
\end{aligned}
\]
\end{lemma}

Combining the fluctuation and bias bounds yields the following localized high-probability estimate for the full estimator error \(g(x)-\nabla f_\mu(x)\).

\begin{corollary}[Localized high-probability deviation of the clipped estimator]
\label{cor:hp-shared-localized}
Assume that Assumptions~\ref{ass:Well-defined}, \ref{ass:sample-smooth}, and
\ref{ass:hp-shared-prob} hold. Fix \(\delta\in(0,1)\) and \(T\ge 1\), and set
$
\lambda:=\log\frac{T}{\delta}.
$
Let \(x\in\R^d\) satisfy
$
f(x)-f_*\le 4\bar\Delta_0.
$
Let \(\tau\ge 4S_\mu\). For \(\ell=1,\dots,M\), let
\[
Y_\ell:=
\frac{F(x+\mu u_\ell;\xi_\ell)-F(x-\mu u_\ell;\xi_\ell)}{2\mu},
\qquad
g(x):=
\frac{d}{M}\sum_{\ell=1}^M \psi_\tau(Y_\ell)u_\ell,
\]
where the pairs \((u_\ell,\xi_\ell)\) are independent across \(\ell\), with
\(u_\ell\sim\Unif(\mathbb S^{d-1})\) and \(\xi_\ell\sim\mathbb P\). Then, with
probability at least \(1-\delta/T\),
\[
\begin{aligned}
\norm{g(x)-\nabla f_\mu(x)}_2
\le\;&
\frac{8d}{\sqrt M}
S_\mu^{p/2}
\tau^{1-\frac p2}
\sqrt{
1+\log\left(\frac{\tau}{S_\mu}\right)
}
\left(1+\sqrt{2\lambda}\right)
+
\frac{8d\tau}{3M}\lambda
+
\frac{2^{2p+1}}{p-1}\,
dS_\mu^p
\tau^{1-p}.
\end{aligned}
\]
\end{corollary}

We now instantiate Corollary~\ref{cor:hp-shared-localized} with the clipping
level used in the algorithm. This choice balances the leading fluctuation
and bias terms up to logarithmic factors and gives the following simplified
localized deviation bound.

\begin{corollary}[Appendix form of Theorem~\ref{thm:tuned-dev}]
\label{cor:hp-shared-localized-tuned}
Assume that Assumptions~\ref{ass:Well-defined}, \ref{ass:sample-smooth}, and
\ref{ass:hp-shared-prob} hold. Fix \(\delta\in(0,1]\) and \(T\ge 3\), and set
$
\lambda:=\log\frac{T}{\delta}.
$
Let \(x\in\R^d\) satisfy
$
f(x)-f_*\le 4\bar\Delta_0.
$
Suppose that \(M\ge \lambda\), and choose
$
\tau
:=
8S_\mu
\left(\frac{M}{\lambda}\right)^{1/p}.
$
For \(\ell=1,\dots,M\), let
\[
Y_\ell:=
\frac{F(x+\mu u_\ell;\xi_\ell)-F(x-\mu u_\ell;\xi_\ell)}{2\mu},
\qquad
g(x):=
\frac{d}{M}\sum_{\ell=1}^M \psi_\tau(Y_\ell)u_\ell .
\]
Then, with probability at least \(1-\delta/T\),
\[
\norm{g(x)-\nabla f_\mu(x)}_2
\le
\left(
120+\frac{2^{4-p}}{p-1}
\right)
dS_\mu
\left(\frac{\log(T/\delta)}{M}\right)^{\frac{p-1}{p}}
\sqrt{1+\log\left(\frac{M}{\log(T/\delta)}\right)}.
\]
\end{corollary}

We now turn the localized deviation bound of
Corollary~\ref{cor:hp-shared-localized-tuned} into a descent estimate for
the smoothed objective \(f_\mu\). Since the update uses the clipped estimator
\(g(x_t)\) in place of \(\nabla f_\mu(x_t)\), the key step is to combine the
smoothness inequality for \(f_\mu\) with the high-probability control of
\(g(x_t)-\nabla f_\mu(x_t)\). The next proposition gives a deterministic
one-step descent inequality.

\begin{proposition}[One-step descent inequality]
\label{prop:hp-shared-onestep}
Assume that \(f\) is \(L\)-smooth. Let \(x\in\R^d\), let
\(x^+:=x-\alpha g(x)\), and define
$
e(x):=g(x)-\nabla f_\mu(x).
$
If \(\alpha\le \frac{1}{4L}\), then
\[
f_\mu(x^+)
\le
f_\mu(x)
-\frac{\alpha}{2}\norm{\nabla f_\mu(x)}_2^2
+\frac{5\alpha}{4}\norm{e(x)}_2^2.
\]
\end{proposition}

For \(M\in\mathbb N\), \(\mu>0\), \(\delta\in(0,1)\), and \(T\ge 3\), define
$
C_p:=120+\frac{2^{4-p}}{p-1},
$
\[
\eta_0(M,\mu,\delta,T)
:=
C_p\,dS_\mu
\left(\frac{\log(T/\delta)}{M}\right)^{\frac{p-1}{p}}
\sqrt{1+\log\!\left(\frac{M}{\log(T/\delta)}\right)},
\]
and
$
\eta(M,\mu,\delta,T):=\eta_0(M,\mu,\delta,T)+L\mu.
$

\begin{theorem}[High-probability convergence of base RSC-ZO]
\label{thm:hp-shared-clipped}
Assume that Assumptions~\ref{ass:Well-defined}, \ref{ass:sample-smooth}, and
\ref{ass:hp-shared-prob} hold. Let \(\mu>0\),
$
\Delta_0:=f(x_0)-f_*
$
and
$
\bar\Delta_0:=\Delta_0+\frac{L\mu^2}{2}.
$
Let \(\epsilon>0\), \(\delta\in(0,1]\), and \(M\in\mathbb N\), and define
$
T:=\max\left\{3,\left\lceil
\frac{32L\bar\Delta_0}{\epsilon^2}
\right\rceil\right\}.
$
Let \(\{x_t\}_{t\ge 0}\) be generated by Algorithm~\ref{alg:rsc-zo-unified}
with \(\beta=0\), batch size \(M\), stepsize \(\alpha=1/(4L)\), and threshold
$
\tau:=
8S_\mu
\Bigl(\frac{M}{\log(T/\delta)}\Bigr)^{1/p}.$
Suppose that
\[
M\ge \log\frac{T}{\delta},
\qquad
\eta(M,\mu,\delta,T)\le \frac{\epsilon}{4},
\qquad
\epsilon^2\le 32L\bar\Delta_0.
\]
Then, with probability at least \(1-\delta\),
\[
\frac1T\sum_{t=0}^{T-1}\norm{\nabla f(x_t)}_2^2
\le \epsilon^2.
\]
In particular, on the same event,
\[
\min_{0\le t<T}\norm{\nabla f(x_t)}_2\le \epsilon.
\]
\end{theorem}

\begin{remark}[Complexity implied by Theorem~\ref{thm:hp-shared-clipped}]
\label{rem:hp-shared-complexity}
The total number of stochastic function queries is \(2MT\). Choose
\(\mu\asymp \epsilon/(Ld)\). Then \(dL\mu=O(\epsilon)\),
\(L\mu=O(\epsilon/d)\), and
\[
\bar\Delta_0
=
\Delta_0
+
O\!\left(\frac{\epsilon^2}{Ld^2}\right).
\]
Since
\[
T
=
\max\left\{
3,\left\lceil \frac{32L\bar\Delta_0}{\epsilon^2}\right\rceil
\right\},
\]
we have
$
T=O(1+L\Delta_0/\epsilon^2).
$
When \(L\Delta_0/\epsilon^2\gtrsim 1\), this simplifies to
$
T=O(L\Delta_0/\epsilon^2).
$

Up to logarithmic factors, the condition
$
\eta(M,\mu,\delta,T)\le \epsilon/4
$
reduces to
\[
dS_\mu
\left(\frac{\log(T/\delta)}{M}\right)^{\frac{p-1}{p}}
\lesssim
\epsilon.
\]
Moreover, under the above choice of \(\mu\),
\[
dS_\mu
=
\sqrt d\,(\sqrt{L\bar\Delta_0}+\sigma)+dL\mu
=
\sqrt d\,(\sqrt{L\Delta_0}+\sigma)+O(\epsilon).
\]
Therefore it is enough to choose
\[
M
\gtrsim
\log(T/\delta)
\left(
\frac{\sqrt d\,(\sqrt{L\Delta_0}+\sigma)}{\epsilon}
\right)^{\frac{p}{p-1}},
\]
up to logarithmic factors. Combining this with
\(T=O(L\Delta_0/\epsilon^2)\) gives
\[
MT
=
\wtO\!\left(
L\Delta_0
\left(
\sqrt{L\Delta_0}+\sigma
\right)^{\frac{p}{p-1}}
\frac{d^{\frac{p}{2(p-1)}}}
{\epsilon^{\frac{3p-2}{p-1}}}
\right).
\]
Thus the total number of stochastic function queries, \(2MT\), has the same
order up to the leading factor \(2\). In particular, at \(p=2\),
\[
MT
=
\wtO\!\left(
L\Delta_0
(\sqrt{L\Delta_0}+\sigma)^2
d\,\epsilon^{-4}
\right),
\]
so the method has the classical \(d\,\epsilon^{-4}\) dimension--accuracy
dependence up to logarithmic and problem-dependent factors.
\end{remark}

\begin{remark}[Trivial regime]
\label{rem:hp-shared-trivial}
The condition \(\epsilon^2\le 32L\bar\Delta_0\) isolates the regime where the
iteration bound above is needed. If \(\epsilon^2>32L\bar\Delta_0\), then
Lemma~\ref{lem:hp-shared-grad-subopt} gives
\[
\|\nabla f(x_0)\|_2^2
\le
2L(f(x_0)-f_*)
=
2L\Delta_0
\le
2L\bar\Delta_0
<
\epsilon^2.
\]
Hence the desired stationarity conclusion already holds at the initial point.
\end{remark}

The proofs of the lemmas, corollaries, and theorem in this section are collected
in Appendix~\ref{proofo}.
\subsection{Proofs}
\label{proofo}

\begin{lemma}[Gradient bound from suboptimality]
\label{lem:hp-shared-grad-subopt}
Assume that $f$ is $L$-smooth and bounded below by $f_*$. Then, for every $x\in\R^d$,
\[
\norm{\nabla f(x)}_2^2
\le
2L\bigl(f(x)-f_*\bigr).
\]
\end{lemma}

\begin{proof}
For any $x\in\R^d$, $L$-smoothness gives
\[
f\Bigl(x-\frac1L\nabla f(x)\Bigr)
\le
f(x)-\frac{1}{2L}\norm{\nabla f(x)}_2^2.
\]
Since $f_*\le f(y)$ for every $y$, we obtain
\[
f_*
\le
f\Bigl(x-\frac1L\nabla f(x)\Bigr)
\le
f(x)-\frac{1}{2L}\norm{\nabla f(x)}_2^2,
\]
which rearranges to the claim.
\end{proof}

\begin{theorem}[{\cite[Theorem~5]{pmlr-v258-mirzaei25a}}]\label{Berst-thm}
Suppose that the $X_i$ are $m$ independent mean-zero random variables with values in a Hilbert space $H$, satisfying $\|X_i\| \le c$. Then for $\delta \in (0,1)$,
\[
\Pr\left\{
\left\|\sum_i X_i\right\|
>
\sqrt{\sum_i \mathbb{E}\|X_i\|^2}
\left(1+\sqrt{2\ln(1/\delta)}\right)
+
\frac{4c}{3}\ln(1/\delta)
\right\}
< \delta .
\]
\end{theorem}
The following lemma is the finite-dimensional form used in our analysis.

\begin{lemma}[Vector Bernstein inequality with variance proxy]
\label{lem:hp-shared-vector-bernstein}
Let $X_1,\dots,X_M$ be independent mean-zero random vectors in $\mathbb{R}^d$.
Assume that
\[
\norm{X_\ell}_2\le R
\quad\text{almost surely for all }\ell,
\]
and
\[
\sum_{\ell=1}^M \mathbb{E}\norm{X_\ell}_2^2 \le V.
\]
Then, for every $\delta\in(0,1)$, with probability at least $1-\delta$,
\[
\left\|\sum_{\ell=1}^M X_\ell\right\|_2
\le
\sqrt{V}\left(1+\sqrt{2\log(1/\delta)}\right)
+
\frac{4R}{3}\log(1/\delta).
\]
\end{lemma}
\begin{proof}
We can apply Theorem~\ref{Berst-thm} with
\(H=\mathbb{R}^d\), \(m=M\), and \(c=R\). The theorem gives
\[
\left\|\sum_{\ell=1}^M X_\ell\right\|_2
\le
\sqrt{\sum_{\ell=1}^M \mathbb{E}\norm{X_\ell}_2^2}
\left(1+\sqrt{2\log(1/\delta)}\right)
+
\frac{4R}{3}\log(1/\delta)
\]
with probability at least \(1-\delta\). Since
$
\sum_{\ell=1}^M \mathbb{E}\norm{X_\ell}_2^2 \le V,$
the result follows.
\end{proof}

\begin{lemma}[Smoothing bias]
\label{lem:hp-shared-smoothing-bias}
Assume that \(f\) is \(L\)-smooth. Then, for every \(x\in\R^d\) and every \(\mu>0\),
\[
\abs{f_\mu(x)-f(x)}\le \frac{L\mu^2}{2}
\quad \text{and} \quad
\norm{\nabla f_\mu(x)-\nabla f(x)}_2\le L\mu.
\]
\end{lemma}

\begin{proof}
By \(L\)-smoothness, for every \(v\in\mathbb B^d\),
\[
\left|
f(x+\mu v)-f(x)-\langle \nabla f(x),\mu v\rangle
\right|
\le \frac{L\mu^2}{2}\|v\|^2
\le \frac{L\mu^2}{2}.
\]
Taking expectation over \(v\sim\mathrm{Unif}(\mathbb B^d)\) and using
\(\mathbb E[v]=0\), we obtain
\[
|f_\mu(x)-f(x)|\le \frac{L\mu^2}{2}.
\]

Next, since \(f\) is \(L\)-smooth,
\[
\nabla f_\mu(x)=\E_{v\sim\Unif(\Ball)}[\nabla f(x+\mu v)].
\]
Therefore,
\[
\norm{\nabla f_\mu(x)-\nabla f(x)}_2
\le
\E_{v\sim\Unif(\Ball)}
\norm{\nabla f(x+\mu v)-\nabla f(x)}_2
\le
L\mu\,\E\norm{v}_2
\le L\mu,
\]
since \(\norm{v}_2\le 1\) almost surely for \(v\in\Ball\).
\end{proof}

\begin{proof}[Proof of Lemma~\ref{lem:hp-shared-clipped-bias-abstract}]
We have
$
\psi_\tau(z)=\operatorname{sgn}(z)\min\{|z|,\tau\}.$
Therefore
\[
|Z-\psi_\tau(Z)|
=
\bigl(|Z|-\tau\bigr)_+,
\]
where \(a_+:=\max\{a,0\}\).

For any nonnegative random variable \(X\) and any \(a>0\),
\[
\E[(X-a)_+]= \E\bigg[\int_{a}^\infty 1_{\{X>t\}} \,dt\bigg]=\int_a^\infty \Pb(X>t)\,dt.
\]
Applying this with \(X=|Z|\) and \(a=\tau\), we obtain
\[
\E\bigl[|Z-\psi_\tau(Z)|\bigr]
=
\int_\tau^\infty \Pb(|Z|>t)\,dt.
\]
Since \(\tau\ge 4S\), the assumed tail bound is valid for every \(t\ge \tau\), so
\[
\E\bigl[|Z-\psi_\tau(Z)|\bigr]
\le
2^{2p+1}S^p\int_\tau^\infty t^{-p}\,dt.
\]
Since \(p>1\), we have
$
\int_\tau^\infty t^{-p}\,dt
=
\frac{\tau^{1-p}}{p-1}.$
Hence
\[
\E\bigl[|Z-\psi_\tau(Z)|\bigr]
\le
\frac{2^{2p+1}}{p-1}\,S^p\tau^{1-p}.
\]
\end{proof}

\begin{proof}[Proof of Lemma~\ref{lem:hp-shared-directional-tail}]
Fix \(x\), and condition on \(x\). By the fundamental theorem of calculus,
\[
\frac{F(x+\mu u;\xi)-F(x-\mu u;\xi)}{2\mu}
=
\frac12\int_{-1}^1
\bigl\langle
\nabla F(x+\mu s u;\xi),u
\bigr\rangle
\,ds.
\]
Similarly,
\[
\frac{f(x+\mu u)-f(x-\mu u)}{2\mu}
=
\frac12\int_{-1}^1
\bigl\langle
\nabla f(x+\mu s u),u
\bigr\rangle
\,ds.
\]
Hence
\[
\frac{F(x+\mu u;\xi)-F(x-\mu u;\xi)}{2\mu}-D_\mu f(x,u)
=
\bigl\langle
\nabla F(x;\xi)-\nabla f(x),u
\bigr\rangle
+
R(x,u,\xi),
\]
where
\begin{align*}
R(x,u,\xi)
&:=
\frac12\int_{-1}^1
\bigl\langle
\nabla F(x+\mu s u;\xi)-\nabla F(x;\xi),u
\bigr\rangle
\,ds \\
&\qquad
-
\frac12\int_{-1}^1
\bigl\langle
\nabla f(x+\mu s u)-\nabla f(x),u
\bigr\rangle
\,ds.
\end{align*}
By Assumption~\ref{ass:sample-smooth} and Remark~\ref{rem:f-smooth}, both
\(F(\cdot;\xi)\) and \(f\) are \(L\)-smooth. Therefore
\[
|R(x,u,\xi)|
\le
\frac12\int_{-1}^1 L\mu |s|\,ds
+
\frac12\int_{-1}^1 L\mu |s|\,ds
=
L\mu.
\]
Thus
\[
\left|
\frac{F(x+\mu u;\xi)-F(x-\mu u;\xi)}{2\mu}-D_\mu f(x,u)
\right|
\le
|\langle \zeta,u\rangle|
+
L\mu\quad
\text{where }\quad
\zeta:=\nabla F(x;\xi)-\nabla f(x).
\]

Define
\[
v:=
\begin{cases}
\zeta/\|\zeta\|_2, & \zeta\neq 0,\\[4pt]
e_1, & \zeta=0,
\end{cases}
\]
where \(e_1\) is the first canonical basis vector, and let
\[
V:=\sqrt d\,|U_1|,
\quad \text{where}\quad
U\sim \Unif(\Sph),
\]is chosen independently of \(\xi\). Since \(u\sim\Unif(\Sph)\), by rotational
invariance we have, conditionally on \(x\) and \(\xi\),
\[
\sqrt d\,|\langle v,u\rangle|
\overset{d}=V.
\]
Since \(|\langle \zeta,u\rangle|=\frac{\|\zeta\|_2}{\sqrt d}\,\sqrt d\,|\langle v,u\rangle|\), it follows that,
conditionally on \(x\) and \(\xi\),
\[
|\langle \zeta,u\rangle|
\overset{d}=
\frac{\|\zeta\|_2}{\sqrt d}\,V.
\]

Moreover, since \(1<p\le 2\),
$
\E[V^p]
=
d^{p/2}\E[|U_1|^p]
\le
d^{p/2}\bigl(\E[U_1^2]\bigr)^{p/2}
=
1.$
Therefore, for every \(r>0\),
\begin{align*}
\Pr\!\left(
|\langle \zeta,u\rangle|>r
\;\middle|\;
x
\right)
&=
\E\!\left[
\Pr\!\left(
|\langle \zeta,u\rangle|>r
\;\middle|\;
x,\xi
\right)
\middle|\;
x
\right] \\
&=
\E\!\left[
\Pr\!\left(
\frac{\|\zeta\|_2}{\sqrt d}V>r
\;\middle|\;
x,\xi
\right)
\middle|\;
x
\right] \\
&=
\E\!\left[
\Pr\!\left(
\|\zeta\|_2>\frac{r\sqrt d}{V}
\;\middle|\;
x,V
\right)
\middle|\;
x
\right].
\end{align*}
Since \(V\) is independent of \(\xi\), Assumption~\ref{ass:hp-shared-prob} yields
\[
\Pr\!\left(
\|\zeta\|_2>\frac{r\sqrt d}{V}
\;\middle|\;
x,V
\right)
\le
\left(\frac{\sigma V}{r\sqrt d}\right)^p.
\]
Hence
\begin{align*}
\Pr\!\left(
|\langle \zeta,u\rangle|>r
\;\middle|\;
x
\right)
&\le
\E\!\left[
\left(\frac{\sigma V}{r\sqrt d}\right)^p
\middle|\;
x
\right] \\
&=
\left(\frac{\sigma}{r\sqrt d}\right)^p \E[V^p]
\le
\frac{\sigma^p}{r^p d^{p/2}}.
\end{align*}

Finally, for every \(s\ge 1\), since \(L\mu\le L\mu s\), we have
\begin{align*}
&\Pr\!\left(
\left|
\frac{F(x+\mu u;\xi)-F(x-\mu u;\xi)}{2\mu}-D_\mu f(x,u)
\right|
>
\left(
\frac{\sigma}{\sqrt d}
+
L\mu
\right)s
\;\middle|\;
x
\right) \\
&\qquad \le
\Pr\!\left(
|\langle \zeta,u\rangle| + L\mu
>
\left(
\frac{\sigma}{\sqrt d}
+
L\mu
\right)s
\;\middle|\;
x
\right) \\
&\qquad \le
\Pr\!\left(
|\langle \zeta,u\rangle|
>
\frac{\sigma}{\sqrt d}\,s
\;\middle|\;
x
\right) \\
&\qquad \le
\left(
\frac{\sigma}{(\sigma/\sqrt d)s\,\sqrt d}
\right)^p
=
s^{-p}.
\end{align*}
\end{proof}
\begin{proof}[Proof of Lemma~\ref{lem:hp-shared-directional-signal-tail}]
Fix \(x\in\R^d\) such that \(f(x)-f_*\le 4\bar\Delta_0\), and let \(u\sim\Unif(\mathbb S^{d-1})\).
By the fundamental theorem of calculus,
\[
D_\mu f(x,u)
=
\frac12\int_{-1}^1 \bigl\langle \nabla f(x+\mu s u),u\bigr\rangle\,ds
=
\langle \nabla f(x),u\rangle + R_\mu(x,u),
\]
where
\[
R_\mu(x,u)
:=
\frac12\int_{-1}^1
\bigl\langle \nabla f(x+\mu s u)-\nabla f(x),u\bigr\rangle\,ds.
\]
Since \(f\) is \(L\)-smooth by Remark~\ref{rem:f-smooth},
\[
|R_\mu(x,u)|
\le
\frac12\int_{-1}^1 L\mu |s|\,ds
=
\frac{L\mu}{2}.
\]

Set \(g:=\nabla f(x)\). By Lemma~\ref{lem:hp-shared-grad-subopt} and the localization
\(f(x)-f_*\le 4\bar\Delta_0\),
\[
\|g\|_2^2 \le 2L(f(x)-f_*) \le 8L\bar\Delta_0.
\]
If \(g=0\), then
$
|D_\mu f(x,u)|\le \frac{L\mu}{2}\le L\mu s$ for every  $s\ge 1,$
so the claim is immediate.

Assume now that \(g\neq 0\), and define
$
v:=\frac{g}{\|g\|_2},$ where $
Z:=\sqrt d\,|\langle v,u\rangle|.$
By rotational invariance,
\[
Z \overset{d}= \sqrt d\,|U_1|,
\quad \text{where}\quad
U\sim\Unif(\mathbb S^{d-1}).
\]
Hence, since \(1<p\le 2\),
\[
\E[Z^p]
=
d^{p/2}\E[|U_1|^p]
\le
d^{p/2}\bigl(\E[U_1^2]\bigr)^{p/2}
=
1.
\]
Therefore, for every \(r>0\), Markov's inequality gives
\[
\Pb_u\!\left(|\langle g,u\rangle|>r\right)
=
\Pb_u\!\left(Z>\frac{r\sqrt d}{\|g\|_2}\right)
\le
\left(\frac{\|g\|_2}{r\sqrt d}\right)^p \E[Z^p]
\le
\left(\frac{\|g\|_2}{r\sqrt d}\right)^p.
\]

Now let \(s\ge 1\). Since \(\|g\|_2\le \sqrt{8L\bar\Delta_0}\), we obtain
\[
\Pb_u\!\left(
|\langle g,u\rangle|>\frac{\sqrt{8L\bar\Delta_0}}{\sqrt d}\,s
\right)
\le s^{-p}.
\]
Using
$
|D_\mu f(x,u)|
\le
|\langle g,u\rangle| + \frac{L\mu}{2}$
and the fact that \(\frac{L\mu}{2}\le L\mu s\) for \(s\ge 1\), we conclude
\begin{align*}
\Pb_u\!\left(
|D_\mu f(x,u)|>
\left(\frac{\sqrt{8L\bar\Delta_0}}{\sqrt d}+L\mu\right)s
\right)
&\le
\Pb_u\!\left(
|\langle g,u\rangle|>
\frac{\sqrt{8L\bar\Delta_0}}{\sqrt d}\,s
\right) \\
&\le s^{-p}.
\end{align*}
\end{proof}

\begin{proof}[Proof of Lemma~\ref{lem:hp-shared-directional-total-tail}]
By Lemma~\ref{lem:hp-shared-directional-signal-tail}, for every $s\ge 1$,
\[
\Pb\bigg(\abs{D_\mu f(x,u)}>3\bigg(\frac{\sqrt{L\bar\Delta_0}+\sigma}{\sqrt d}+L\mu\bigg)s\bigg)\le s^{-p},
\]
because
\[
\frac{\sqrt{8L\bar\Delta_0}}{\sqrt d}+L\mu
\le
3\left(\frac{\sqrt{L\bar\Delta_0}+\sigma}{\sqrt d}+L\mu\right).
\]
By Lemma~\ref{lem:hp-shared-directional-tail}, for every $s\ge 1$,
\[
\Pb(\abs{Y-D_\mu f(x,u)}>\bigg(\frac{\sqrt{L\bar\Delta_0}+\sigma}{\sqrt d}+L\mu\bigg)s)\le s^{-p}.
\]
Therefore, for every $s\ge 1$,
\[
\begin{aligned}
\Pb\bigg(\abs{Y} &> 4\bigg(\frac{\sqrt{L\bar\Delta_0}+\sigma}{\sqrt d}+L\mu\bigg)s\bigg) \\
&\le \Pb\bigg(\abs{D_\mu f(x,u)} > 3\bigg(\frac{\sqrt{L\bar\Delta_0}+\sigma}{\sqrt d}+L\mu\bigg)s\bigg) \\
&\quad + \Pb\bigg(\abs{Y - D_\mu f(x,u)} > \bigg(\frac{\sqrt{L\bar\Delta_0}+\sigma}{\sqrt d}+L\mu\bigg)s\bigg) \\
&\le 2s^{-p}.
\end{aligned}
\]
Let
$
S:=\frac{\sqrt{L\bar\Delta_0}+\sigma}{\sqrt d}+L\mu.$
If $t\ge 4S$, then taking
$
s:=\frac{t}{4S}\ge 1$
in the preceding bound gives
\[
\Pb(|Y|>t)
\le 2s^{-p}
=
2\left(\frac{4S}{t}\right)^p
=
2^{2p+1}S^p t^{-p}.
\]
\end{proof}

\begin{proof}[Proof of Lemma~\ref{lem:hp-shared-clipping-bias}]
Let
$
S:=
\frac{\sqrt{L\bar\Delta_0}+\sigma}{\sqrt d}+L\mu .$
Since
\[
\E[Y\mid u]
=
\frac{f(x+\mu u)-f(x-\mu u)}{2\mu}
=
D_\mu f(x,u),
\]
we have
\[
d\,\E[Yu]
=
d\,\E[D_\mu f(x,u)u]
=
\nabla f_\mu(x).
\]
Therefore,
\[
d\,\E[\psi_\tau(Y)u]-\nabla f_\mu(x)
=
d\,\E[(\psi_\tau(Y)-Y)u].
\]
It holds that
\[
\Bigl\|d\,\E[\psi_\tau(Y)u]-\nabla f_\mu(x)\Bigr\|_2
\le
d\,\E\bigl[|\psi_\tau(Y)-Y|\bigr].
\]

By Lemma~\ref{lem:hp-shared-directional-total-tail}, the random variable \(Y\) satisfies
\[
\forall t\ge 4S,
\quad
\Pb(|Y|>t)\le 2^{2p+1}S^p t^{-p}.
\]
Since the lemma assumes \(\tau\ge 4S\), Lemma~\ref{lem:hp-shared-clipped-bias-abstract} applies and yields
\[
\E[|Y-\psi_\tau(Y)|]
\le
\frac{2^{2p+1}}{p-1}S^p\tau^{1-p}.
\]
Substituting this into the previous bound gives
\[
\Bigl\|d\,\E[\psi_\tau(Y)u]-\nabla f_\mu(x)\Bigr\|_2
\le
\frac{2^{2p+1}}{p-1}\,
d\,S^p\tau^{1-p}.
\]
Recalling the definition of \(S\), this is exactly
\[
\Bigl\|d\,\E[\psi_\tau(Y)u]-\nabla f_\mu(x)\Bigr\|_2
\le
\frac{2^{2p+1}}{p-1}\,
d
\left(
\frac{\sqrt{L\bar\Delta_0}+\sigma}{\sqrt d}
+
L\mu
\right)^p
\tau^{1-p}.
\]

For the equivalent form, let
$
g(x):=
\frac{d}{M}\sum_{\ell=1}^M \psi_\tau(Y_\ell)u_\ell$
where \((Y_\ell,u_\ell)\) are i.i.d. copies of \((Y,u)\). Then
\[
\E[g(x)]
=
d\,\E[\psi_\tau(Y)u],
\]
and therefore
\[
\E[g(x)]-\nabla f_\mu(x)
=
d\,\E[\psi_\tau(Y)u]-\nabla f_\mu(x).
\]
Thus the same bound holds for
$B:=\E[g(x)]-\nabla f_\mu(x).$
\end{proof}

\begin{proof}[Proof of Lemma~\ref{lem:hp-shared-clipped-second-moment}]We have
$
\psi_\tau(z)=\operatorname{sgn}(z)\min\{|z|,\tau\},$
so
$
\psi_\tau(Z)^2=\min\{Z^2,\tau^2\}.$
Using the layer-cake formula,
\[
\E[\psi_\tau(Z)^2]
=2\int_0^\infty t\,\Pb(|\psi_\tau(Z)|>t)\,dt=
2\int_0^\tau t\,\Pb(|Z|>t)\,dt.
\]
We split the integral at \(4S\):
\[
\E[\psi_\tau(Z)^2]
=
2\int_0^{4S} t\,\Pb(|Z|>t)\,dt
+
2\int_{4S}^\tau t\,\Pb(|Z|>t)\,dt.
\]
Since \(\Pb(|Z|>t)\le 1\), the first term is bounded by
$
2\int_0^{4S} t\,dt = 16S^2.$
For the second term, use the assumed tail bound:
\[
2\int_{4S}^\tau t\,\Pb(|Z|>t)\,dt
\le
2^{2p+2}S^p\int_{4S}^\tau t^{1-p}\,dt.
\]
Now, for \(4S\le t\le \tau\) and \(1<p\le 2\),
\[
t^{1-p}=t^{-1}t^{2-p}\le t^{-1}\tau^{2-p},
\]
hence
\[
\int_{4S}^\tau t^{1-p}\,dt
\le
\tau^{2-p}\int_{4S}^\tau \frac{dt}{t}
=
\tau^{2-p}\log\!\left(\frac{\tau}{4S}\right).
\]
Therefore,
\[
\E[\psi_\tau(Z)^2]
\le
16S^2 + 2^{2p+2}S^p\tau^{2-p}\log\!\left(\frac{\tau}{4S}\right).
\]
Since \(\tau\ge 4S\) and \(1<p\le 2\),
$
S^2\le S^p\tau^{2-p},$
so
\[
16S^2 \le 16S^p\tau^{2-p}.
\]
Also,
\[
\log\!\left(\frac{\tau}{4S}\right)\le \log\!\left(\frac{\tau}{S}\right).
\]
Thus
\[
\E[\psi_\tau(Z)^2]
\le
\Bigl(16+2^{2p+2}\log(\tau/S)\Bigr)S^p\tau^{2-p}.
\]
Finally, since \(1<p\le 2\), we have \(2^{2p+2}\le 64\). Hence
\[
\E[\psi_\tau(Z)^2]
\le
64\,S^p\tau^{2-p}\bigl(1+\log(\tau/S)\bigr).
\]
\end{proof}

\begin{proof}[Proof of Lemma~\ref{lem:hp-shared-empirical-fluctuation}]
Let
$
S:=
\frac{\sqrt{L\bar\Delta_0}+\sigma}{\sqrt d}+L\mu.$
For each \(\ell\), define
\[
X_\ell:=
\frac{d}{M}
\Bigl(
\widetilde Y_\ell u_\ell-\E[\widetilde Y_\ell u_\ell]
\Bigr).
\]
Then
$
A=\sum_{\ell=1}^M X_\ell$
and \(X_1,\dots,X_M\) are independent mean-zero random vectors.

Since \(\|u_\ell\|_2=1\) almost surely and \(|\widetilde Y_\ell|\le \tau\), we have
\[
\|X_\ell\|_2
\le
\frac{d}{M}
\Bigl(
|\widetilde Y_\ell|
+
\|\E[\widetilde Y_\ell u_\ell]\|_2
\Bigr)
\le
\frac{2d\tau}{M}.
\]
Thus the almost-sure bound in Lemma~\ref{lem:hp-shared-vector-bernstein} holds with
$
R:=\frac{2d\tau}{M}.$ Next, using
$
\E\|Z-\E Z\|_2^2\le \E\|Z\|_2^2,
$
we obtain
\[
\sum_{\ell=1}^M \E\|X_\ell\|_2^2
\le
\frac{d^2}{M}\E[\widetilde Y^2],
\]
where \(\widetilde Y=\psi_\tau(Y)\), and \(Y\) has the same distribution as in
Lemma~\ref{lem:hp-shared-directional-total-tail}. By Lemma~\ref{lem:hp-shared-directional-total-tail},
\[
\forall t\ge 4S,
\quad
\Pb(|Y|>t)\le 2^{2p+1}S^p t^{-p}.
\]
Since \(\tau\ge 4S\), Lemma~\ref{lem:hp-shared-clipped-second-moment} gives
\[
\E[\widetilde Y^2]
\le
64S^p\tau^{2-p}
\left(1+\log(\tau/S)\right).
\]
Therefore
\[
V:=
\sum_{\ell=1}^M \E\|X_\ell\|_2^2
\le
\frac{64d^2}{M}
S^p\tau^{2-p}
\left(1+\log(\tau/S)\right).
\]

Applying Lemma~\ref{lem:hp-shared-vector-bernstein} with failure probability \(\rho\), we obtain, with probability at least \(1-\rho\),
\[
\|A\|_2
\le
\sqrt V\left(1+\sqrt{2\log(1/\rho)}\right)
+
\frac{4R}{3}\log(1/\rho).
\]
Using the bounds on \(V\) and \(R\), this gives
\[
\|A\|_2
\le
\frac{8d}{\sqrt M}
S^{p/2}\tau^{1-\frac p2}
\sqrt{1+\log(\tau/S)}
\left(1+\sqrt{2\log(1/\rho)}\right)
+
\frac{8d\tau}{3M}\log(1/\rho).
\]

\end{proof}

\begin{proof}[Proof of Corollary~\ref{cor:hp-shared-localized}]
We have
\[
g(x)-\nabla f_\mu(x)
=
\underbrace{g(x)-\E[g(x)]}_{A}
+
\underbrace{\E[g(x)]-\nabla f_\mu(x)}_{B}.
\]
Apply Lemma~\ref{lem:hp-shared-empirical-fluctuation} with failure probability
$
\rho:=\frac{\delta}{T}.$
Since \(\log(1/\rho)=\log(T/\delta)=\lambda\), we obtain, with probability at least
\(1-\delta/T\),
\[
\begin{aligned}
\norm{A}_2
\le\;&
\frac{8d}{\sqrt M}
\left(
\frac{\sqrt{L\bar\Delta_0}+\sigma}{\sqrt d}
+
L\mu
\right)^{p/2}
\tau^{1-\frac p2}
\\
&\quad\times
\sqrt{
1+
\log\left(
\frac{\tau}{
\frac{\sqrt{L\bar\Delta_0}+\sigma}{\sqrt d}+L\mu
}
\right)
}
\left(1+\sqrt{2\lambda}\right)
+
\frac{8d\tau}{3M}\lambda.
\end{aligned}
\]
On the other hand, Lemma~\ref{lem:hp-shared-clipping-bias} gives deterministically
\[
\norm{B}_2
\le
\frac{2^{2p+1}}{p-1}\,
d
\left(
\frac{\sqrt{L\bar\Delta_0}+\sigma}{\sqrt d}
+
L\mu
\right)^p
\tau^{1-p}.
\]
Combining the two bounds by the triangle inequality gives the result.
\end{proof}

\begin{proof}[Proof of Corollary~\ref{cor:hp-shared-localized-tuned} (Appendix form of Theorem~\ref{thm:tuned-dev})]
Let
$
S:=
\frac{\sqrt{L\bar\Delta_0}+\sigma}{\sqrt d}+L\mu.$
Since \(T\ge 3\) and \(\delta\in(0,1]\), we have
$
\lambda=\log(T/\delta)\ge \log 3>1.
$
Moreover, since \(M\ge \lambda\),
\[
\tau=8S\left(\frac{M}{\lambda}\right)^{1/p}
\ge 8S\ge 4S,
\]
so Corollary~\ref{cor:hp-shared-localized} applies. From Corollary~\ref{cor:hp-shared-localized}, with probability at least \(1-\delta/T\),
\[
\begin{aligned}
\norm{g(x)-\nabla f_\mu(x)}_2
\le\;&
\frac{8d}{\sqrt M}S^{p/2}\tau^{1-\frac p2}
\sqrt{1+\log(\tau/S)}
\left(1+\sqrt{2\lambda}\right)
\\
&+
\frac{8d\tau}{3M}\lambda
+
\frac{2^{2p+1}}{p-1}dS^p\tau^{1-p}.
\end{aligned}
\]
We bound the three terms separately.

For the first term, substituting
$
\tau=8S\left(\frac{M}{\lambda}\right)^{1/p}$
gives
$
\frac{8d}{\sqrt M}S^{p/2}\tau^{1-\frac p2}
=
8^{2-\frac p2}
dS
M^{-\frac{p-1}{p}}
\lambda^{-\frac{2-p}{2p}}.
$
Since \(\lambda\ge 1\),
$
1+\sqrt{2\lambda}
\le
(1+\sqrt 2)\sqrt{\lambda}.
$
Therefore,
\[
\begin{aligned}
&\frac{8d}{\sqrt M}S^{p/2}\tau^{1-\frac p2}
\sqrt{1+\log(\tau/S)}
\left(1+\sqrt{2\lambda}\right)
\\
&\le
8^{2-\frac p2}(1+\sqrt 2)
dS
\left(\frac{\lambda}{M}\right)^{\frac{p-1}{p}}
\sqrt{1+\log\left(8\left(\frac{M}{\lambda}\right)^{1/p}\right)}.
\end{aligned}
\]
Since \(1<p\le 2\), we have \(8^{2-p/2}\le 8^{3/2}<23\). Also, because \(M\ge \lambda\),
\[
1+\log\left(8\left(\frac{M}{\lambda}\right)^{1/p}\right)
\le
(1+\log 8)
\left(
1+\log\left(\frac{M}{\lambda}\right)
\right).
\]
Thus the first term is bounded by
\[
98\,dS
\left(\frac{\lambda}{M}\right)^{\frac{p-1}{p}}
\sqrt{1+\log\left(\frac{M}{\lambda}\right)}.
\]

For thelinear term,
\[
\frac{8d\tau}{3M}\lambda
=
\frac{64}{3}dS
\left(\frac{\lambda}{M}\right)^{\frac{p-1}{p}}
\le
22\,dS
\left(\frac{\lambda}{M}\right)^{\frac{p-1}{p}}
\sqrt{1+\log\left(\frac{M}{\lambda}\right)}.
\]

For the clipping-bias term,
\[
\frac{2^{2p+1}}{p-1}dS^p\tau^{1-p}
=
\frac{2^{4-p}}{p-1}
dS
\left(\frac{\lambda}{M}\right)^{\frac{p-1}{p}}.
\]
Since \(\sqrt{1+\log(M/\lambda)}\ge 1\), this is at most
\[
\frac{2^{4-p}}{p-1}
dS
\left(\frac{\lambda}{M}\right)^{\frac{p-1}{p}}
\sqrt{1+\log\left(\frac{M}{\lambda}\right)}.
\]

Combining the three bounds gives
\[
\norm{g(x)-\nabla f_\mu(x)}_2
\le
\left(
120+\frac{2^{4-p}}{p-1}
\right)
dS
\left(\frac{\lambda}{M}\right)^{\frac{p-1}{p}}
\sqrt{1+\log\left(\frac{M}{\lambda}\right)}.
\]

\end{proof}

\begin{proof}[Proof of Proposition~\ref{prop:hp-shared-onestep}]
Since $f$ is $L$-smooth, its spherical smoothing $f_\mu$ is also $L$-smooth. Hence
\[
f_\mu(x^+)
\le
f_\mu(x)+\ip{\nabla f_\mu(x)}{x^+-x}+\frac{L}{2}\norm{x^+-x}_2^2.
\]
Using $x^+-x=-\alpha g(x)$, we obtain
\[
f_\mu(x^+)
\le
f_\mu(x)-\alpha\ip{\nabla f_\mu(x)}{g(x)}+\frac{L\alpha^2}{2}\norm{g(x)}_2^2.
\]
Now write
\[
g(x)=\nabla f_\mu(x)+e(x),
\]
so that
\[
\ip{\nabla f_\mu(x)}{g(x)}
=
\norm{\nabla f_\mu(x)}_2^2+\ip{\nabla f_\mu(x)}{e(x)}.
\]
Moreover,
\[
\norm{g(x)}_2^2
=
\norm{\nabla f_\mu(x)+e(x)}_2^2
\le
2\norm{\nabla f_\mu(x)}_2^2+2\norm{e(x)}_2^2,
\]
and, by Young's inequality,
\[
\abs{\ip{\nabla f_\mu(x)}{e(x)}}
\le
\norm{\nabla f_\mu(x)}_2\,\norm{e(x)}_2
\le
\frac14\norm{\nabla f_\mu(x)}_2^2+\norm{e(x)}_2^2.
\]
Substituting these bounds into the previous inequality yields
\[
f_\mu(x^+)
\le
f_\mu(x)
-\alpha\Bigl(\frac34-L\alpha\Bigr)\norm{\nabla f_\mu(x)}_2^2
+\alpha(1+L\alpha)\norm{e(x)}_2^2.
\]
If $\alpha\le \frac{1}{4L}$, then
\[
\frac34-L\alpha\ge \frac12
\qquad\text{and}\qquad
1+L\alpha\le \frac54.
\]
Therefore,
\[
f_\mu(x^+)
\le
f_\mu(x)-\frac{\alpha}{2}\norm{\nabla f_\mu(x)}_2^2+\frac{5\alpha}{4}\norm{e(x)}_2^2.
\]
\end{proof}

\begin{proof}[Proof of Theorem~\ref{thm:hp-shared-clipped}]
Set
$
\lambda:=\log\frac{T}{\delta}.
$
Since \(T\ge 3\) and \(\delta\in(0,1]\), we have
$
\lambda\ge \log 3>1.
$

We first define the filtration generated by the algorithm. Let
 $
\mathcal F_{-1}:=\sigma(x_0),$
and for \(t\ge 0\), define
\[
\mathcal F_t
:=
\sigma\!\left(
x_0,
\{u_{s,\ell},\xi_{s,\ell}:0\le s\le t,\ 1\le \ell\le M\}
\right).
\]
Then \(x_t\) is \(\mathcal F_{t-1}\)-measurable, while \(g_t\) is
\(\mathcal F_t\)-measurable. Moreover, conditionally on \(\mathcal F_{t-1}\),
the fresh samples
\[
\{u_{t,\ell},\xi_{t,\ell}:1\le \ell\le M\}
\]
are independent of the past and have the same distribution.

For each \(t=0,\dots,T-1\), define the localization event
\[
\mathcal L_t
:=
\Bigl\{
f(x_s)-f_*\le 4\bar\Delta_0
\ \text{for all } s=0,\dots,t
\Bigr\}.
\]
Since \(x_0,\dots,x_t\) are \(\mathcal F_{t-1}\)-measurable, we have
$
\mathcal L_t\in \mathcal F_{t-1}.$
Also define the estimator-good event
\[
\mathcal E_t
:=
\Bigl\{
\norm{g_t-\nabla f_\mu(x_t)}_2
\le
\eta_0(M,\mu,\delta,T)
\Bigr\}.
\]

On the event \(\mathcal L_t\), the point \(x_t\) satisfies
$
f(x_t)-f_*\le 4\bar\Delta_0.$
Therefore, conditionally on \(\mathcal F_{t-1}\), we may apply
Corollary~\ref{cor:hp-shared-localized-tuned} to the fixed point \(x_t\),
using the fresh samples at iteration \(t\). Hence, almost surely,
\[
\mathbf 1_{\mathcal L_t}
\Pb(\mathcal E_t^c\mid \mathcal F_{t-1})
\le
\mathbf 1_{\mathcal L_t}\frac{\delta}{T}.
\]
It follows that
\[
\begin{aligned}
\Pb(\mathcal E_t^c\cap \mathcal L_t)
&=
\E\!\left[
\mathbf 1_{\mathcal E_t^c}\mathbf 1_{\mathcal L_t}
\right]
\\
&=
\E\!\left[
\E\!\left[
\mathbf 1_{\mathcal E_t^c}\mathbf 1_{\mathcal L_t}
\mid \mathcal F_{t-1}
\right]
\right]
\\
&=
\E\!\left[
\mathbf 1_{\mathcal L_t}
\Pb(\mathcal E_t^c\mid \mathcal F_{t-1})
\right]
\\
&\le
\E\!\left[
\mathbf 1_{\mathcal L_t}\frac{\delta}{T}
\right]
\\
&\le
\frac{\delta}{T}.
\end{aligned}
\]
By the union bound,
\[
\Pb\left(
\bigcup_{t=0}^{T-1}(\mathcal E_t^c\cap \mathcal L_t)
\right)
\le
\delta.
\]
Define 
$
\mathcal G
:=
\bigcap_{t=0}^{T-1}
\left(\mathcal L_t^c \cup \mathcal E_t\right).$
Since
$
\mathcal G^c
=
\bigcup_{t=0}^{T-1}(\mathcal E_t^c\cap \mathcal L_t),
$
it holds that
\[
\Pb(\mathcal G)\ge 1-\delta.
\]

We work on the good event \(\mathcal G\)
 for the rest of the proof.
We first prove by induction that the iterates remain localized. Clearly
\[
f(x_0)-f_*=\Delta_0\le \bar\Delta_0\le 4\bar\Delta_0,
\]
so \(\mathcal L_0\) holds. Suppose that \(\mathcal L_t\) holds for some
\(t\le T-2\). Then, on the good event, \(\mathcal E_s\) holds for every
\(s=0,\dots,t\), since \(\mathcal L_t\) implies \(\mathcal L_s\) for all
\(s\le t\). Proposition~\ref{prop:hp-shared-onestep} gives, for
every \(s=0,\dots,t\),
\[
f_\mu(x_{s+1})
\le
f_\mu(x_s)
-\frac{\alpha}{2}\norm{\nabla f_\mu(x_s)}_2^2
+\frac{5\alpha}{4}\eta_0(M,\mu,\delta,T)^2.
\]
Dropping the negative term and summing from \(s=0\) to \(t\), we obtain
\[
f_\mu(x_{t+1})-f_*
\le
f_\mu(x_0)-f_*
+
\frac{5\alpha}{4}(t+1)\eta_0(M,\mu,\delta,T)^2.
\]
We have  \(t+1\le T\), and using Lemma \ref{lem:hp-shared-smoothing-bias}, 
$
f_\mu(x_0)-f_*
\le
f(x_0)-f_*+\frac{L\mu^2}{2}
=
\bar\Delta_0$, then
\[
f_\mu(x_{t+1})-f_*
\le
\bar\Delta_0
+
\frac{5\alpha T}{4}\eta_0(M,\mu,\delta,T)^2.
\]
Now \(\alpha=\frac{1}{4L}\), \(\eta_0\le \eta\le \epsilon/4\), and
\[
T
=
\max\left\{3,\left\lceil\frac{32L\bar\Delta_0}{\epsilon^2}\right\rceil\right\}
\le
\frac{96L\bar\Delta_0}{\epsilon^2}.
\]
Indeed, if \(a:=32L\bar\Delta_0/\epsilon^2\), then the assumption
\(\epsilon^2\le 32L\bar\Delta_0\) gives \(a\ge 1\), and hence
\[
\max\{3,\lceil a\rceil\}\le 3a
=
\frac{96L\bar\Delta_0}{\epsilon^2}.
\]
Therefore,
$
\frac{5\alpha T}{4}\eta_0^2
\le
\frac{5}{16L}
\cdot
\frac{96L\bar\Delta_0}{\epsilon^2}
\cdot
\frac{\epsilon^2}{16}
=
\frac{15}{8}\bar\Delta_0.$
Consequently,
\[
f_\mu(x_{t+1})-f_*
\le
\frac{23}{8}\bar\Delta_0.
\]
Using Lemma \ref{lem:hp-shared-smoothing-bias},
\[
f(x_{t+1})-f_*
\le
f_\mu(x_{t+1})-f_*+\frac{L\mu^2}{2}
\le
\frac{23}{8}\bar\Delta_0+\bar\Delta_0
=
\frac{31}{8}\bar\Delta_0
<
4\bar\Delta_0.
\]
Thus \(\mathcal L_{t+1}\) holds. By induction, \(\mathcal L_t\) holds for all
\(t=0,\dots,T-1\). Hence, on the good event, \(\mathcal E_t\) holds for all
\(t=0,\dots,T-1\).

We now prove the stationarity bound. Summing the one-step descent inequality
over \(t=0,\dots,T-1\), and using \(\norm{e_t}_2\le \eta_0(M,\mu,\delta,T)\), gives
\[
\frac{\alpha}{2}
\sum_{t=0}^{T-1}
\norm{\nabla f_\mu(x_t)}_2^2
\le
f_\mu(x_0)-f_*
+
\frac{5\alpha T}{4}\eta_0(M,\mu,\delta,T)^2.
\]
Since \(f_\mu(x_0)-f_*\le \bar\Delta_0\), it holds that
\[
\frac1T
\sum_{t=0}^{T-1}
\norm{\nabla f_\mu(x_t)}_2^2
\le
\frac{2\bar\Delta_0}{\alpha T}
+
\frac52\eta_0(M,\mu,\delta,T)^2.
\]
Using \(\alpha=\frac{1}{4L}\) and
$
T\ge \frac{32L\bar\Delta_0}{\epsilon^2},$
we get
\[
\frac{2\bar\Delta_0}{\alpha T}
=
\frac{8L\bar\Delta_0}{T}
\le
\frac{\epsilon^2}{4}.
\]
Therefore,
\[
\frac1T
\sum_{t=0}^{T-1}
\norm{\nabla f_\mu(x_t)}_2^2
\le
\frac{\epsilon^2}{4}
+
\frac52\eta_0(M,\mu,\delta,T)^2.
\]

Using Lemma \ref{lem:hp-shared-smoothing-bias},
$
\norm{\nabla f(x_t)}_2
\le
\norm{\nabla f_\mu(x_t)}_2+L\mu.$
Hence
\[
\norm{\nabla f(x_t)}_2^2
\le
2\norm{\nabla f_\mu(x_t)}_2^2+2L^2\mu^2.
\]
Averaging over \(t=0,\dots,T-1\), we obtain
\begin{align*}
\frac1T
\sum_{t=0}^{T-1}
\norm{\nabla f(x_t)}_2^2
&\le
2\left(
\frac{\epsilon^2}{4}
+
\frac52\eta_0^2
\right)
+
2L^2\mu^2\\
&\le \frac{\epsilon^2}{2}
+
5\eta_0^2
+
2L^2\mu^2.
\end{align*}

Since
$
\eta=\eta_0+L\mu,$
we have
\[
5\eta_0^2+2L^2\mu^2
\le
5(\eta_0+L\mu)^2
=
5\eta^2.
\]
Using the assumption \(\eta\le \epsilon/4\), we conclude that
\[
\frac1T
\sum_{t=0}^{T-1}
\norm{\nabla f(x_t)}_2^2
\le
\frac{\epsilon^2}{2}
+
5\frac{\epsilon^2}{16}
=
\frac{13}{16}\epsilon^2
\le
\epsilon^2.
\]

\end{proof}

\section{Proofs for the High-Probability Analysis of Momentum RSC-ZO}
\label{app:momentum}

This appendix proves the high-probability guarantee for the momentum variant of
RSC-ZO. The proof builds on the localized scalar-clipping estimates used for
the base method and adds a momentum-tracking argument. The main quantity to
control is the error between the momentum vector \(m_t\) and the smoothed
gradient \(\nabla f_\mu(x_t)\), while keeping the iterates in the localized
region where the one-step estimator bounds apply.

\paragraph{Roadmap.}
We first decompose each clipped estimator into its conditional mean, centered
fluctuation, and clipping bias. We then control the geometrically weighted
fluctuation accumulated by the momentum recursion. After removing this
fluctuation, the remaining tracking error satisfies a stable deterministic
recurrence. Combining this recurrence with the Lyapunov function defined below
and a first-exit localization step gives the convergence bound.

\subsection{Setup and notation}

We work under Assumptions~\ref{ass:Well-defined}, \ref{ass:sample-smooth}, and
\ref{ass:hp-shared-prob}. Write
\[
  \Delta_0:=f(x_0)-f_*,
  \qquad
  \bar\Delta_0:=\Delta_0+\frac{L\mu^2}{2},
  \qquad
  S_\mu:=\frac{\sqrt{L\bar\Delta_0}+\sigma}{\sqrt d}+L\mu .
\]
Thus \(dS_\mu=\sqrt d(\sqrt{L\bar\Delta_0}+\sigma)+dL\mu\).

Let \(\beta\in[1/2,1)\), \(T\ge4\), \(\delta\in(0,1]\), and
\(M_0,M\in\mathbb N\). Define
\(\lambda:=\log(2T/\delta)\), \(\lambda_0:=\log(2/\delta)\), and assume
\(M\ge(1-\beta)\lambda\), \(M_0\ge\lambda_0\). The running and warm-start
thresholds are
\[
  \tau
  :=
  8S_\mu
  \left(\frac{M}{(1-\beta)\lambda}\right)^{1/p},
  \qquad
  \tau_0
  :=
  8S_\mu
  \left(\frac{M_0}{\lambda_0}\right)^{1/p}.
\]
We also write \(q:=(1-\beta)\lambda/M\) and \(q_0:=\lambda_0/M_0\). Let
\[
  U
  :=
  C_U dS_\mu q^{\frac{p-1}{p}}
  \sqrt{1+\log\frac{M}{(1-\beta)\lambda}},
  \qquad
  B
  :=
  C_B dS_\mu q^{\frac{p-1}{p}},
\]
and
\[
  G
  :=
  C_G dS_\mu q_0^{\frac{p-1}{p}}
  \sqrt{1+\log\frac{M_0}{\lambda_0}},
  \qquad
  H_\mu:=\sqrt{2L\Delta_0}+L\mu .
\]
Here \(C_U,C_B,C_G>0\) are sufficiently large absolute constants depending only
on the constants in the one-step estimator bounds below.

The momentum method is initialized by
\[
  g_0=G_\mu(x_0;M_0,\tau_0),
  \qquad
  m_0=g_0,
  \qquad
  x_1=x_0-\alpha m_0.
\]
For \(t\ge1\), it forms
\[
  g_t
  =
  G_\mu(x_t;M,\tau)
  =
  \frac dM
  \sum_{\ell=1}^M
  \psi_\tau(Y_{t,\ell})u_{t,\ell},
  \qquad
  Y_{t,\ell}
  :=
  \frac{
    F(x_t+\mu u_{t,\ell};\xi_{t,\ell})
    -
    F(x_t-\mu u_{t,\ell};\xi_{t,\ell})
  }{2\mu},
\]
and updates
\[
  m_t=\beta m_{t-1}+(1-\beta)g_t,
  \qquad
  x_{t+1}=x_t-\alpha m_t.
\]

Let \(\mathcal F_{-1}:=\sigma(x_0)\). Let \(\mathcal F_0\) contain the
warm-start batch, and for \(t\ge1\), let \(\mathcal F_t\) contain the warm-start
batch and all running batches up to time \(t\). Then \(x_t\) is
\(\mathcal F_{t-1}\)-measurable for every \(t\ge1\), while \(g_t,m_t,x_{t+1}\)
are \(\mathcal F_t\)-measurable.

For \(t\ge1\), define
\[
  \mathcal L_t
  :=
  \left\{
    f(x_s)-f_*\le4\bar\Delta_0
    \text{ for all }s=0,1,\ldots,t
  \right\}.
\]
Then \(\mathcal L_t\in\mathcal F_{t-1}\). Define
\[
  a_t:=g_t-\E[g_t\mid\mathcal F_{t-1}],
  \qquad
  b_t:=\E[g_t\mid\mathcal F_{t-1}]-\nabla f_\mu(x_t),
\]
so that \(g_t=\nabla f_\mu(x_t)+a_t+b_t\). Let
\(e_t:=m_t-\nabla f_\mu(x_t)\). Define
\[
  \widehat z_t
  :=
  \sum_{s=1}^t
  \beta^{t-s}(1-\beta)a_s,
  \qquad
  z_t
  :=
  \sum_{s=1}^t
  \beta^{t-s}(1-\beta)\mathbf 1_{\mathcal L_s}a_s.
\]
On \(\mathcal L_t\), we have \(\widehat z_t=z_t\). Finally, define
\[
  u_t:=e_t-\widehat z_t,
  \qquad
  \Phi_t
  :=
  f_\mu(x_t)
  +
  \frac{2\alpha}{1-\beta}\|u_t\|_2^2,
  \qquad t\ge1.
\]

\subsection{One-step estimator bounds used by the momentum proof}

We use the following localized one-step consequences of the scalar-clipping
analysis.

\begin{lemma}[Localized clipped-estimator bounds]
\label{lem:momentum-onestep-estimator}
There exist absolute constants \(C_{\rm bias},C_{\rm var}>0\) such that the
following holds. Let \(x\) be \(\mathcal H\)-measurable and assume
\(f(x)-f_*\le4\bar\Delta_0\). Let \((u_\ell,\xi_\ell)_{\ell=1}^M\) be
conditionally independent of \(\mathcal H\), with
\(u_\ell\sim\Unif(\mathbb S^{d-1})\) and \(\xi_\ell\sim\mathcal P\). Define
\[
  Y_\ell
  :=
  \frac{
    F(x+\mu u_\ell;\xi_\ell)-F(x-\mu u_\ell;\xi_\ell)
  }{2\mu},
  \qquad
  Z_\ell:=\psi_\tau(Y_\ell)u_\ell,
  \qquad
  g:=\frac dM\sum_{\ell=1}^M Z_\ell .
\]
If \(\bar g:=\E[g\mid\mathcal H]\) and \(\tau\ge8S_\mu\), then
\[
  \|\bar g-\nabla f_\mu(x)\|_2
  \le
  C_{\rm bias}dS_\mu^p\tau^{1-p}.
\]
Moreover, with
\(W_\ell:=(d/M)(Z_\ell-\E[Z_\ell\mid\mathcal H])\), one has, conditionally on
\(\mathcal H\),
\[
  \|W_\ell\|_2\le\frac{2d\tau}{M}
  \quad\text{a.s.},
  \qquad
  \sum_{\ell=1}^M
  \E[\|W_\ell\|_2^2\mid\mathcal H]
  \le
  C_{\rm var}
  \frac{d^2}{M}
  S_\mu^p\tau^{2-p}
  \left(1+\log\frac{\tau}{S_\mu}\right).
\]
\end{lemma}

\begin{proof}
Under localization, the shared-randomness directional estimator satisfies
\[
  \Pr(|Y_\ell|>r\mid\mathcal H)
  \lesssim
  \left(\frac{S_\mu}{r}\right)^p,
  \qquad r\ge S_\mu .
\]
Therefore
\[
  \|\bar g-\nabla f_\mu(x)\|_2
  \le
  d\,\E[|Y_\ell|\mathbf 1\{|Y_\ell|>\tau\}\mid\mathcal H]
  \lesssim
  dS_\mu^p\tau^{1-p}.
\]
Also \(\|Z_\ell\|_2\le\tau\), and hence
\[
  \|W_\ell\|_2
  \le
  \frac dM
  \left(
    \|Z_\ell\|_2+\E[\|Z_\ell\|_2\mid\mathcal H]
  \right)
  \le
  \frac{2d\tau}{M}.
\]
Finally,
\[
  \sum_{\ell=1}^M\E[\|W_\ell\|_2^2\mid\mathcal H]
  \le
  \frac{4d^2}{M}
  \E[\psi_\tau(Y_\ell)^2\mid\mathcal H],
\]
and weak-\(L_p\) tail integration gives
\[
  \E[\psi_\tau(Y_\ell)^2\mid\mathcal H]
  \lesssim
  S_\mu^p\tau^{2-p}
  \left(1+\log\frac{\tau}{S_\mu}\right).
\]
Absorbing numerical constants into \(C_{\rm bias},C_{\rm var}\) proves the
claim.
\end{proof}

\begin{lemma}[Vector Freedman]
\label{lem:momentum-freedman}
Let \((X_i)_{i=1}^n\) be a martingale-difference sequence in a Hilbert space
with respect to a filtration \((\mathcal H_i)_{i=0}^n\). Assume
\(\|X_i\|_2\le R\) a.s. for all \(i\), and
\[
  \sum_{i=1}^n
  \E[\|X_i\|_2^2\mid\mathcal H_{i-1}]
  \le V
  \quad\text{a.s.}
\]
Then there is an absolute constant \(C_F>0\) such that, for every \(\eta>0\),
with probability at least \(1-e^{-\eta}\),
\[
  \left\|
    \sum_{i=1}^n X_i
  \right\|_2
  \le
  C_F(\sqrt{V\eta}+R\eta).
\]
\end{lemma}

\subsection{Stopped weighted fluctuation bound}

We first control the stochastic fluctuation accumulated by the momentum
recursion. The stopped process \(z_t\) is a geometrically weighted sum of
centered batch errors on the localized event.

\begin{lemma}[Stopped weighted fluctuation]
\label{lem:momentum-stopped-fluctuation}
If \(C_U\) is sufficiently large, then for every \(t\in\{1,\ldots,T-1\}\),
\[
  \Pr(\|z_t\|_2>U)\le\frac{\delta}{2T}.
\]
Consequently, the event
\[
  \mathcal E
  :=
  \left\{
    \max_{1\le t\le T-1}\|z_t\|_2\le U
  \right\}
\]
satisfies \(\Pr(\mathcal E)\ge1-\delta/2\).
\end{lemma}

\begin{proof}
Fix \(t\in\{1,\ldots,T-1\}\). For \(s\le t\), set
\(Z_{s,\ell}:=\psi_\tau(Y_{s,\ell})u_{s,\ell}\),
\(\bar Z_s:=\E[Z_{s,1}\mid\mathcal F_{s-1}]\), and
\(W_{s,\ell}:=(d/M)(Z_{s,\ell}-\bar Z_s)\). Then
\(a_s=\sum_{\ell=1}^M W_{s,\ell}\), and therefore
\[
  z_t
  =
  \sum_{s=1}^t
  \sum_{\ell=1}^M
  \beta^{t-s}(1-\beta)
  \mathbf 1_{\mathcal L_s}
  W_{s,\ell}.
\]
Reveal the samples \((s,\ell)\) in lexicographic order and define
\[
  \mathcal H_{s,\ell}
  :=
  \sigma\!\left(
    \mathcal F_{s-1},
    \{(u_{s,j},\xi_{s,j}):1\le j\le\ell\}
  \right),
  \qquad
  \mathcal H_{s,0}:=\mathcal F_{s-1}.
\]
Since \(\mathbf 1_{\mathcal L_s}\) is \(\mathcal F_{s-1}\)-measurable and
\(\E[W_{s,\ell}\mid\mathcal H_{s,\ell-1}]=0\), the variables
\[
  X_{s,\ell}^{(t)}
  :=
  \beta^{t-s}(1-\beta)\mathbf 1_{\mathcal L_s}W_{s,\ell}
\]
form a Hilbert-valued martingale-difference array. On \(\mathcal L_s\),
Lemma~\ref{lem:momentum-onestep-estimator} gives
\[
  \|X_{s,\ell}^{(t)}\|_2
  \le
  \frac{2d\tau(1-\beta)}{M}
  =:R,
\]
and
\[
\begin{aligned}
  \sum_{s=1}^t\sum_{\ell=1}^M
  \E[\|X_{s,\ell}^{(t)}\|_2^2\mid\mathcal H_{s,\ell-1}]
  &\le
  C_{\rm var}
  \frac{d^2(1-\beta)}{M}
  S_\mu^p\tau^{2-p}
  \left(1+\log\frac{\tau}{S_\mu}\right)
  =:V .
\end{aligned}
\]
Here we used
\((1-\beta)^2\sum_{j=0}^{t-1}\beta^{2j}\le1-\beta\). Applying
Lemma~\ref{lem:momentum-freedman} with \(\eta=\lambda\), with probability at
least \(1-e^{-\lambda}=1-\delta/(2T)\),
\[
  \|z_t\|_2
  \le
  C_F(\sqrt{V\lambda}+R\lambda).
\]
It remains to simplify the two terms. Since
\(\tau/S_\mu=8q^{-1/p}\) and \(q=(1-\beta)\lambda/M\le1\),
\[
  R\lambda
  =
  \frac{2d\tau(1-\beta)\lambda}{M}
  =
  16dS_\mu q^{\frac{p-1}{p}},
\]
and
\[
  \sqrt{V\lambda}
  \le
  C dS_\mu q^{\frac{p-1}{p}}
  \sqrt{
    1+\log\frac{M}{(1-\beta)\lambda}
  }.
\]
Choosing \(C_U\) sufficiently large gives
\(\|z_t\|_2\le U\) with probability at least \(1-\delta/(2T)\). A union bound
over \(t=1,\ldots,T-1\) gives \(\Pr(\mathcal E)\ge1-\delta/2\).
\end{proof}

\subsection{Bias and drift bounds}

\begin{lemma}[One-step clipping bias]
\label{lem:momentum-onestep-bias}
On \(\mathcal L_t\), for every \(t\ge1\), \(\|b_t\|_2\le B\). Consequently,
on \(\mathcal L_t\),
\[
  \left\|
    \sum_{s=1}^t
    \beta^{t-s}(1-\beta)b_s
  \right\|_2
  \le B.
\]
\end{lemma}

\begin{proof}
On \(\mathcal L_t\), each \(x_s\), \(1\le s\le t\), belongs to the localized
region. Lemma~\ref{lem:momentum-onestep-estimator} gives
\[
  \|b_s\|_2
  \le
  C_{\rm bias}dS_\mu^p\tau^{1-p}
  \le
  C dS_\mu q^{\frac{p-1}{p}}
  \le B,
\]
where we used \(\tau=8S_\mu q^{-1/p}\) and chose \(C_B\) sufficiently large.
The weighted bound follows from
\(\sum_{s=1}^t\beta^{t-s}(1-\beta)\le1\).
\end{proof}

\begin{lemma}[Drift recurrence]
\label{lem:momentum-drift-recurrence}
For every \(t\ge2\),
\[
  u_t
  =
  \beta u_{t-1}
  +(1-\beta)b_t
  +r_t,
  \qquad
  r_t
  :=
  \beta(\nabla f_\mu(x_{t-1})-\nabla f_\mu(x_t)).
\]
Moreover, if \(\beta\in[1/2,1)\), then on \(\mathcal L_t\),
\[
  \|u_t\|_2^2
  \le
  \frac{1+\beta}{2}\|u_{t-1}\|_2^2
  +
  2(1-\beta)B^2
  +
  \frac{3\beta^2L^2\alpha^2}{1-\beta}
  \|m_{t-1}\|_2^2.
\]
\end{lemma}

\begin{proof}
Using \(m_t=\beta m_{t-1}+(1-\beta)g_t\) and
\(g_t=\nabla f_\mu(x_t)+a_t+b_t\), we get
\[
\begin{aligned}
  e_t
  &=
  \beta e_{t-1}
  +(1-\beta)a_t
  +(1-\beta)b_t
  +
  \beta(\nabla f_\mu(x_{t-1})-\nabla f_\mu(x_t)).
\end{aligned}
\]
Since \(\widehat z_t=\beta\widehat z_{t-1}+(1-\beta)a_t\), subtracting gives
the stated recurrence for \(u_t\).

Now work on \(\mathcal L_t\). Let
\(q_t:=\beta u_{t-1}+(1-\beta)b_t\). Since \(\|b_t\|_2\le B\),
\[
  \|q_t\|_2^2
  \le
  \beta\|u_{t-1}\|_2^2+(1-\beta)B^2.
\]
Young's inequality with \(\rho=(1-\beta)/(2\beta)\) gives
\[
  \|q_t+r_t\|_2^2
  \le
  \frac{1+\beta}{2}\|u_{t-1}\|_2^2
  +
  2(1-\beta)B^2
  +
  \frac{3\beta}{1-\beta}\|r_t\|_2^2,
\]
where we used \(\beta\ge1/2\). By \(L\)-smoothness of \(f_\mu\),
\[
  \|r_t\|_2
  \le
  \beta L\|x_t-x_{t-1}\|_2
  =
  \beta L\alpha\|m_{t-1}\|_2.
\]
Thus the last term is at most
\(3\beta^3L^2\alpha^2\|m_{t-1}\|_2^2/(1-\beta)\), which is bounded by the
displayed term since \(\beta<1\). This proves the claim.
\end{proof}

\subsection{Descent and Lyapunov inequality}

\begin{lemma}[Momentum descent]
\label{lem:momentum-descent}
Let \(t\ge1\) and assume \(\alpha\le1/(2L)\). Then
\[
  f_\mu(x_{t+1})
  \le
  f_\mu(x_t)
  -
  \frac\alpha2\|\nabla f_\mu(x_t)\|_2^2
  -
  \frac\alpha4\|m_t\|_2^2
  +
  \alpha\|u_t\|_2^2
  +
  \alpha\|\widehat z_t\|_2^2.
\]
In particular, on \(\mathcal L_t\cap\mathcal E\),
\[
  f_\mu(x_{t+1})
  \le
  f_\mu(x_t)
  -
  \frac\alpha2\|\nabla f_\mu(x_t)\|_2^2
  -
  \frac\alpha4\|m_t\|_2^2
  +
  \alpha\|u_t\|_2^2
  +
  \alpha U^2.
\]
\end{lemma}

\begin{proof}
By \(L\)-smoothness of \(f_\mu\) and \(x_{t+1}=x_t-\alpha m_t\),
\[
  f_\mu(x_{t+1})
  \le
  f_\mu(x_t)
  -
  \alpha\langle\nabla f_\mu(x_t),m_t\rangle
  +
  \frac{L\alpha^2}{2}\|m_t\|_2^2.
\]
Using
\[
  -\langle a,b\rangle
  =
  \frac12\|a-b\|_2^2
  -
  \frac12\|a\|_2^2
  -
  \frac12\|b\|_2^2
\]
with \(a=\nabla f_\mu(x_t)\), \(b=m_t\), and
\(m_t-\nabla f_\mu(x_t)=u_t+\widehat z_t\), we obtain
\[
  f_\mu(x_{t+1})
  \le
  f_\mu(x_t)
  -
  \frac\alpha2\|\nabla f_\mu(x_t)\|_2^2
  -
  \frac\alpha2(1-L\alpha)\|m_t\|_2^2
  +
  \frac\alpha2\|u_t+\widehat z_t\|_2^2.
\]
Since \(\alpha\le1/(2L)\), \(1-L\alpha\ge1/2\), and
\(\|u_t+\widehat z_t\|_2^2\le2\|u_t\|_2^2+2\|\widehat z_t\|_2^2\). This gives
the first claim. On \(\mathcal L_t\), \(\widehat z_t=z_t\), and on
\(\mathcal E\), \(\|z_t\|_2\le U\), which gives the second claim.
\end{proof}

\begin{theorem}[One-step Lyapunov inequality]
\label{thm:momentum-lyapunov-step}
Assume \(\beta\in[1/2,1)\) and
\[
  \alpha
  \le
  \min
  \left\{
    \frac1{2L},
    \frac{1-\beta}{4\sqrt3\,\beta L}
  \right\}.
\]
Then, on \(\mathcal L_{t+1}\cap\mathcal E\), for every \(t=1,\ldots,T-2\),
\[
  \Phi_{t+1}
  \le
  \Phi_t
  -
  \frac\alpha2\|\nabla f_\mu(x_t)\|_2^2
  -
  \frac\alpha8\|m_t\|_2^2
  +
  \alpha(U^2+4B^2).
\]
\end{theorem}

\begin{proof}
Since \(\mathcal L_{t+1}\subseteq\mathcal L_t\),
Lemma~\ref{lem:momentum-descent} gives
\[
  f_\mu(x_{t+1})
  \le
  f_\mu(x_t)
  -
  \frac\alpha2\|\nabla f_\mu(x_t)\|_2^2
  -
  \frac\alpha4\|m_t\|_2^2
  +
  \alpha\|u_t\|_2^2
  +
  \alpha U^2.
\]
Also, Lemma~\ref{lem:momentum-drift-recurrence} applied at time \(t+1\) gives
\[
  \|u_{t+1}\|_2^2
  \le
  \frac{1+\beta}{2}\|u_t\|_2^2
  +
  2(1-\beta)B^2
  +
  \frac{3\beta^2L^2\alpha^2}{1-\beta}\|m_t\|_2^2.
\]
Subtracting \(\|u_t\|_2^2\), multiplying by \(2\alpha/(1-\beta)\), and adding
to the previous display gives
\[
  \Phi_{t+1}
  \le
  \Phi_t
  -
  \frac\alpha2\|\nabla f_\mu(x_t)\|_2^2
  -
  \alpha
  \left(
    \frac14-\frac{6\beta^2L^2\alpha^2}{(1-\beta)^2}
  \right)\|m_t\|_2^2
  +
  \alpha U^2
  +
  4\alpha B^2.
\]
The stepsize condition gives
\(6\beta^2L^2\alpha^2/(1-\beta)^2\le1/8\). Hence the coefficient of
\(\|m_t\|_2^2\) is at least \(1/8\), proving the claim.
\end{proof}

\subsection{Warm start and initialization}

Define
\[
  \mathcal E_0
  :=
  \{\|g_0-\nabla f_\mu(x_0)\|_2\le G\},
  \qquad
  \mathcal I
  :=
  \{\Phi_1-f_*\le2\bar\Delta_0\}.
\]

\begin{lemma}[Warm-start deviation]
\label{lem:momentum-warm-start}
If \(C_G\) is sufficiently large, then
\[
  \Pr(\mathcal E_0)\ge1-\frac\delta2.
\]
\end{lemma}

\begin{proof}
The point \(x_0\) satisfies \(f(x_0)-f_*\le\bar\Delta_0\le4\bar\Delta_0\).
Applying the localized deviation bound from
Corollary~\ref{cor:hp-shared-localized}, with failure probability \(\delta/2\),
batch size \(M_0\), and threshold
\[
  \tau_0
  =
  8S_\mu
  \left(\frac{M_0}{\lambda_0}\right)^{1/p},
\]
gives the claim after increasing \(C_G\) if necessary.
\end{proof}

\begin{lemma}[Initialization]
\label{lem:momentum-initialization}
Assume \(\alpha\le1/(4L)\). If
\[
  \frac{5\alpha}{4}G^2
  +
  \frac{2\alpha}{1-\beta}
  \left(
    \beta((1+L\alpha)G+L\alpha H_\mu)
    +(1-\beta)B
  \right)^2
  \le \bar\Delta_0,
  \tag{Init}
  \label{eq:momentum-init-condition}
\]
then \(\mathcal E_0\subseteq\mathcal I\). Consequently,
\(\Pr(\mathcal I)\ge1-\delta/2\).
\end{lemma}

\begin{proof}
Work on \(\mathcal E_0\). Since \(x_1=x_0-\alpha g_0\), the one-step descent
inequality for \(f_\mu\) gives
\[
  f_\mu(x_1)
  \le
  f_\mu(x_0)
  -
  \frac\alpha2\|\nabla f_\mu(x_0)\|_2^2
  +
  \frac{5\alpha}{4}\|g_0-\nabla f_\mu(x_0)\|_2^2.
\]
Thus \(f_\mu(x_1)-f_*\le\bar\Delta_0+(5\alpha/4)G^2\). By
\eqref{eq:momentum-init-condition}, this is at most \(2\bar\Delta_0\). Hence
\[
  f(x_1)-f_*
  \le
  f_\mu(x_1)-f_*+\frac{L\mu^2}{2}
  \le
  3\bar\Delta_0
  <
  4\bar\Delta_0,
\]
so \(\mathcal L_1\) holds on \(\mathcal E_0\).

By smoothness, \(\|\nabla f(x_0)\|_2^2\le2L\Delta_0\). Since
\(\|\nabla f_\mu(x_0)-\nabla f(x_0)\|_2\le L\mu\), we have
\(\|\nabla f_\mu(x_0)\|_2\le H_\mu\), and therefore
\(\|g_0\|_2\le H_\mu+G\). Moreover,
\[
  \|g_0-\nabla f_\mu(x_1)\|_2
  \le
  G+L\alpha\|g_0\|_2
  \le
  (1+L\alpha)G+L\alpha H_\mu.
\]
Since \(m_1=\beta g_0+(1-\beta)g_1\),
\[
  u_1
  =
  m_1-\nabla f_\mu(x_1)-(1-\beta)a_1
  =
  \beta(g_0-\nabla f_\mu(x_1))+(1-\beta)b_1.
\]
Because \(\mathcal L_1\) holds, Lemma~\ref{lem:momentum-onestep-bias} gives
\(\|b_1\|_2\le B\), and hence
\[
  \|u_1\|_2
  \le
  \beta((1+L\alpha)G+L\alpha H_\mu)
  +(1-\beta)B.
\]
Combining the last bounds,
\[
\begin{aligned}
  \Phi_1-f_*
  &\le
  \bar\Delta_0+\frac{5\alpha}{4}G^2
  +
  \frac{2\alpha}{1-\beta}
  \left(
    \beta((1+L\alpha)G+L\alpha H_\mu)
    +(1-\beta)B
  \right)^2
  \le
  2\bar\Delta_0.
\end{aligned}
\]
Thus \(\mathcal E_0\subseteq\mathcal I\). The probability bound follows from
Lemma~\ref{lem:momentum-warm-start}.
\end{proof}

\subsection{Localized convergence}

\begin{theorem}[Localized high-probability convergence for momentum RSC-ZO]
\label{thm:momentum-localized-convergence}
Assume the hypotheses of Lemma~\ref{lem:momentum-initialization} and
Theorem~\ref{thm:momentum-lyapunov-step}. Suppose moreover that
\[
  (T-1)\alpha(U^2+4B^2)\le\bar\Delta_0.
  \tag{Loc}
  \label{eq:momentum-localization-condition}
\]
Then, with probability at least \(1-\delta\),
\[
  f(x_t)-f_*\le4\bar\Delta_0
  \qquad
  \text{for all }t=1,\ldots,T-1,
\]
and
\[
  \frac1{T-2}
  \sum_{t=1}^{T-2}\|\nabla f(x_t)\|_2^2
  \le
  \frac{8\bar\Delta_0}{\alpha(T-2)}
  +
  4U^2
  +
  16B^2
  +
  2L^2\mu^2.
\]
\end{theorem}

\begin{proof}
By Lemmas~\ref{lem:momentum-initialization} and
\ref{lem:momentum-stopped-fluctuation},
\(\Pr(\mathcal I\cap\mathcal E)\ge1-\delta\). Work on
\(\mathcal I\cap\mathcal E\).

We first prove localization. Suppose localization fails before time \(T\), and
let
\[
  r
  :=
  \min
  \{t\in\{1,\ldots,T-1\}:f(x_t)-f_*>4\bar\Delta_0\}.
\]
The initialization proof gives \(f(x_1)-f_*<4\bar\Delta_0\), so \(r\ge2\) and
\(\mathcal L_{r-1}\) holds. For each \(t=1,\ldots,r-2\), we have
\(\mathcal L_{t+1}\), so Theorem~\ref{thm:momentum-lyapunov-step} applies.
Dropping the negative terms and summing,
\[
  \Phi_{r-1}
  \le
  \Phi_1+(r-2)\alpha(U^2+4B^2)
  \le
  f_*+2\bar\Delta_0+(r-2)\alpha(U^2+4B^2).
\]
By Lemma~\ref{lem:momentum-descent} at time \(r-1\),
\[
  f_\mu(x_r)
  \le
  f_\mu(x_{r-1})
  +
  \alpha\|u_{r-1}\|_2^2
  +
  \alpha U^2
  \le
  \Phi_{r-1}+\alpha U^2,
\]
where we used \(2\alpha/(1-\beta)\ge\alpha\). Hence
\[
  f_\mu(x_r)-f_*
  \le
  2\bar\Delta_0+(T-1)\alpha(U^2+4B^2).
\]
Using \(f(x)\le f_\mu(x)+L\mu^2/2\) and
\eqref{eq:momentum-localization-condition}, we get
\[
  f(x_r)-f_*
  \le
  2\bar\Delta_0
  +(T-1)\alpha(U^2+4B^2)
  +
  \frac{L\mu^2}{2}
  \le
  4\bar\Delta_0,
\]
contradicting the definition of \(r\). Thus \(\mathcal L_{T-1}\) holds.

Now Theorem~\ref{thm:momentum-lyapunov-step} applies for all
\(t=1,\ldots,T-2\). Summing and using \(f_\mu(x_{T-1})\ge f_*\) gives
\[
  \frac\alpha2
  \sum_{t=1}^{T-2}\|\nabla f_\mu(x_t)\|_2^2
  \le
  \Phi_1-f_*+(T-2)\alpha(U^2+4B^2).
\]
Since \(\Phi_1-f_*\le2\bar\Delta_0\),
\[
  \frac1{T-2}
  \sum_{t=1}^{T-2}\|\nabla f_\mu(x_t)\|_2^2
  \le
  \frac{4\bar\Delta_0}{\alpha(T-2)}
  +
  2U^2
  +
  8B^2.
\]
Finally, \(\|\nabla f(x_t)\|_2\le\|\nabla f_\mu(x_t)\|_2+L\mu\), and hence
\[
  \|\nabla f(x_t)\|_2^2
  \le
  2\|\nabla f_\mu(x_t)\|_2^2
  +
  2L^2\mu^2.
\]
Averaging proves the claim.
\end{proof}

\subsection{Complexity consequence}

We first state the complexity in terms of
\(\bar\Delta_0=\Delta_0+L\mu^2/2\). Under the final smoothing choice for
\(\mu\), this recovers the simpler dependence on \(\Delta_0\).

\begin{corollary}[Complexity of momentum RSC-ZO]
\label{cor:momentum-complexity}
Let \(\epsilon\in(0,1)\), choose \(\mu\le\epsilon/(4Ld)\), and assume
\(\epsilon^2\le32L\bar\Delta_0\). Choose
\[
  \alpha=\frac{1-\beta}{16\sqrt3\,L},
  \qquad
  T
  =
  \left\lceil
    512\sqrt3\,
    \frac{L\bar\Delta_0}{(1-\beta)\epsilon^2}
  \right\rceil
  +2.
\]
Assume \(G\le c_0\epsilon\), \(U\le c_1\epsilon\), and
\(B\le c_1\epsilon\), for sufficiently small absolute constants
\(c_0,c_1>0\). Then, with probability at least \(1-\delta\),
\[
  \frac1{T-2}
  \sum_{t=1}^{T-2}\|\nabla f(x_t)\|_2^2
  \le
  \epsilon^2.
\]
Moreover, up to logarithmic factors and the lower constraints
\(M_0\ge\lambda_0\), \(M\ge(1-\beta)\lambda\), it is enough to choose
\[
  M_0
  =
  \widetilde\Theta
  \left(
    \left(
      \frac{
        \sqrt d(\sqrt{L\bar\Delta_0}+\sigma)+dL\mu
      }{\epsilon}
    \right)^{\frac{p}{p-1}}
  \right),
\]
and
\[
  M
  =
  \widetilde\Theta
  \left(
    (1-\beta)
    \left(
      \frac{
        \sqrt d(\sqrt{L\bar\Delta_0}+\sigma)+dL\mu
      }{\epsilon}
    \right)^{\frac{p}{p-1}}
  \right),
\]
where the hidden constants depend only on \(p\). Consequently, the total number
of stochastic function evaluations satisfies
\[
\begin{aligned}
  Q_{\rm MOM}
  &:=
  2M_0+2M(T-1) \\
  &=
  \widetilde O
  \left(
    \left(
      \frac{
        \sqrt d(\sqrt{L\bar\Delta_0}+\sigma)+dL\mu
      }{\epsilon}
    \right)^{\frac{p}{p-1}}
  \right) \\
  &\quad
  +
  \widetilde O
  \left(
    \frac{L\bar\Delta_0}{\epsilon^2}
    \left(
      \frac{
        \sqrt d(\sqrt{L\bar\Delta_0}+\sigma)+dL\mu
      }{\epsilon}
    \right)^{\frac{p}{p-1}}
  \right).
\end{aligned}
\]
Under \(\mu\le\epsilon/(4Ld)\), we have
\[
  \bar\Delta_0
  =
  \Delta_0+O\!\left(\frac{\epsilon^2}{Ld^2}\right),
  \qquad
  dS_\mu
  =
  \sqrt d(\sqrt{L\Delta_0}+\sigma)+O(\epsilon).
\]
Thus
\[
  Q_{\rm MOM}
  =
  \widetilde O
  \left(
    \left(
      \frac{\sqrt d(\sqrt{L\Delta_0}+\sigma)}{\epsilon}
    \right)^{\frac{p}{p-1}}
  \right)
  +
  \widetilde O
  \left(
    \frac{L\Delta_0}{\epsilon^2}
    \left(
      \frac{\sqrt d(\sqrt{L\Delta_0}+\sigma)}{\epsilon}
    \right)^{\frac{p}{p-1}}
  \right).
\]
Thus momentum reduces the running batch size by a factor of order \(1-\beta\),
while preserving the total query complexity up to logarithmic factors and the
additive warm-start cost.
\end{corollary}

\begin{proof}
We verify the hypotheses of
Theorem~\ref{thm:momentum-localized-convergence}. The chosen stepsize satisfies
\[
  \alpha
  =
  \frac{1-\beta}{16\sqrt3\,L}
  \le
  \min
  \left\{
    \frac1{4L},
    \frac{1-\beta}{4\sqrt3\,\beta L}
  \right\}
\]
for all \(\beta\in[1/2,1)\). Also, by smoothness and
\(\epsilon^2\le32L\bar\Delta_0\),
\[
  H_\mu
  =
  \sqrt{2L\Delta_0}+L\mu
  \le
  C\sqrt{L\bar\Delta_0}.
\]
Using \(G\le c_0\epsilon\), \(B\le c_1\epsilon\), and the choice of \(\alpha\),
the initialization condition is bounded by
\[
  C
  \left(
    c_0^2+c_1^2+\frac{1}{16^2\cdot3}
  \right)\bar\Delta_0.
\]
The numerical constant in the stepsize is chosen so that, after taking
\(c_0,c_1\) sufficiently small, the last display is at most
\(\bar\Delta_0\). Similarly, \((T-1)\alpha\le C\bar\Delta_0/\epsilon^2\), and
therefore
\[
  (T-1)\alpha(U^2+4B^2)
  \le
  Cc_1^2\bar\Delta_0
  \le
  \bar\Delta_0
\]
after decreasing \(c_1\). Hence the localization condition holds.

Theorem~\ref{thm:momentum-localized-convergence} gives
\[
  \frac1{T-2}
  \sum_{t=1}^{T-2}\|\nabla f(x_t)\|_2^2
  \le
  \frac{8\bar\Delta_0}{\alpha(T-2)}
  +
  4U^2
  +
  16B^2
  +
  2L^2\mu^2.
\]
The chosen \(T\), the assumptions on \(U,B\), and
\(\mu\le\epsilon/(4Ld)\) make the right-hand side at most \(\epsilon^2\) after
choosing \(c_0,c_1\) small enough.

It remains to solve \(G\le c_0\epsilon\), \(U\le c_1\epsilon\), and
\(B\le c_1\epsilon\). Since
\(dS_\mu=\sqrt d(\sqrt{L\bar\Delta_0}+\sigma)+dL\mu\), the stated choices of
\(M_0\) and \(M\) satisfy these inequalities up to logarithmic factors. Finally,
each two-point direction uses two stochastic function evaluations, so
\(Q_{\rm MOM}=2M_0+2M(T-1)\). Substituting the estimates for \(M_0,M,T\) proves
the complexity bound.
\end{proof}

\begin{remark}[Interpretation]
The running-phase concentration depends on
\[
  \frac{(1-\beta)\lambda}{M}
\]
rather than \(\lambda/M\). Equivalently, the momentum buffer has the same
concentration behavior as an estimator with effective running batch size
\[
  M_{\rm eff}\asymp \frac{M}{1-\beta}.
\]
This explains why the running batch size can be reduced by a factor of order
\(1-\beta\). The reduction is coupled to the stability condition
\(\alpha\lesssim(1-\beta)/L\): choosing \(\beta\) closer to one reduces the
running batch size, but also forces a proportionally smaller stepsize. The
separate warm start is needed because the first momentum vector does not yet
benefit from time averaging.
\end{remark}

\section{Experiments: Clipping Before versus After Aggregation}
\label{sec:experiments}

We now provide a controlled experiment designed to isolate the mechanism behind
scalar clipping in stochastic zeroth-order optimization. The goal is not to
benchmark on a difficult objective, but to test a specific prediction of our
analysis: under heavy-tailed noise, clipping each scalar directional derivative
before aggregation can be substantially more effective than clipping only the
final batched vector estimator.

The latter baseline is important because it is the most direct analogue of
standard gradient clipping in SGD. In first-order stochastic optimization, one
observes a stochastic gradient vector and clips its norm before taking a step.
The corresponding zeroth-order analogue is therefore to first construct the
batched gradient proxy
\[
    g_{\rm raw}
    =
    \frac dM \sum_{\ell=1}^M Y_\ell u_\ell
\]
and then apply vector clipping,
\[
    g_{\rm vec}
    =
    \operatorname{clip}_{r_{\rm vec}}(g_{\rm raw}).
\]
Our method clips earlier, at the level of each scalar directional derivative
\(Y_\ell\), before averaging. We also include the unclipped estimator
\(g_{\rm raw}\), which corresponds to the standard MeZO-style zeroth-order
update without clipping. Thus, the experiment compares three natural choices:
no clipping, clipping after constructing a gradient-like vector as in standard
SGD gradient clipping, and clipping each zeroth-order directional measurement
before aggregation.

\paragraph{Experimental setup.}
We consider the quadratic objective
\[
    f(x) = \frac12 \|x\|_2^2,
\]
observed through the stochastic zeroth-order oracle
\[
    F(x;\zeta)
    =
    \frac12 \|x\|_2^2 + \langle \zeta,x\rangle .
\]
The noise is sparse and heavy-tailed:
\[
    \zeta = s A e_J,
\]
where \(s\) is uniform on \(\{-1,+1\}\), \(J\) is uniform on
\(\{1,\ldots,d\}\), and \(A\) follows a Pareto distribution with tail exponent
\(p\). Thus
\[
    \nabla F(x;\zeta)-\nabla f(x)=\zeta.
\]
For \(p\in(1,2)\), this noise has finite first moment but infinite variance.
The random sign ensures that \(\mathbb E[\zeta]=0\), so the population objective
is exactly \(f(x)=\frac12\|x\|_2^2\). Since the perturbation
\(\langle \zeta,x\rangle\) is linear in \(x\), every sample objective remains
\(1\)-smooth.

For a direction \(u_\ell\), the shared-randomness two-point estimator satisfies
\[
    Y_\ell
    =
    \frac{
    F(x+\mu u_\ell;\zeta_\ell)
    -
    F(x-\mu u_\ell;\zeta_\ell)
    }{2\mu}
    =
    \langle x,u_\ell\rangle
    +
    \langle \zeta_\ell,u_\ell\rangle .
\]
We compare three estimators:
\[
\begin{cases}
\displaystyle
g_{\rm raw}
=
\frac dM \sum_{\ell=1}^M Y_\ell u_\ell,
& \text{raw / MeZO-style estimator}, \\[1.2em]
\displaystyle
g_{\rm vec}
=
\operatorname{clip}_{r_{\rm vec}}(g_{\rm raw}),
& \text{vector clipping after aggregation}, \\[1.2em]
\displaystyle
g_{\rm sc}
=
\frac dM \sum_{\ell=1}^M \psi_\tau(Y_\ell)u_\ell,
& \text{scalar clipping before aggregation}.
\end{cases}
\]

\paragraph{Why this comparison matters.}
The distinction between \(g_{\rm vec}\) and \(g_{\rm sc}\) is central.
Final-vector clipping controls only the norm of the already-aggregated vector.
Indeed, when \(\|g_{\rm raw}\|_2>r_{\rm vec}\),
\[
    \operatorname{clip}_{r_{\rm vec}}(g_{\rm raw})
    =
    r_{\rm vec}\frac{g_{\rm raw}}{\|g_{\rm raw}\|_2},
\]
so the direction of \(g_{\rm raw}\) is preserved. Therefore, if a few extreme
directional derivatives \(Y_\ell\) dominate the sum, final-vector clipping can
make the step smaller but cannot repair the corrupted direction.

In contrast, scalar clipping acts before aggregation:
\[
    Y_\ell \mapsto \psi_\tau(Y_\ell).
\]
Thus each directional contribution is bounded before it enters the average. This
prevents a single heavy-tailed directional derivative from dominating the
estimated direction.

We emphasize that clipping each individual vector \(dY_\ell u_\ell\) before
averaging would be essentially equivalent to scalar clipping, since
\[
    \|dY_\ell u_\ell\|_2 = d|Y_\ell|.
\]
The meaningful baseline is therefore clipping the final batched estimator, which
is the natural analogue of standard gradient clipping.

\paragraph{Fair tuning.}
For each method, we tune the stepsize using \(6\) validation seeds. For
final-vector clipping, we additionally tune the vector clipping radius
\(r_{\rm vec}\), while for scalar clipping we tune the scalar threshold
\(\tau\). We then fix the selected hyperparameters and evaluate all methods on
\(20\) independent seeds. The reported curves and tables are computed only from
these evaluation seeds, not from the validation seeds used for tuning.

\begin{figure}[t]
    \centering
    \safeincludegraphics[width=0.62\linewidth]{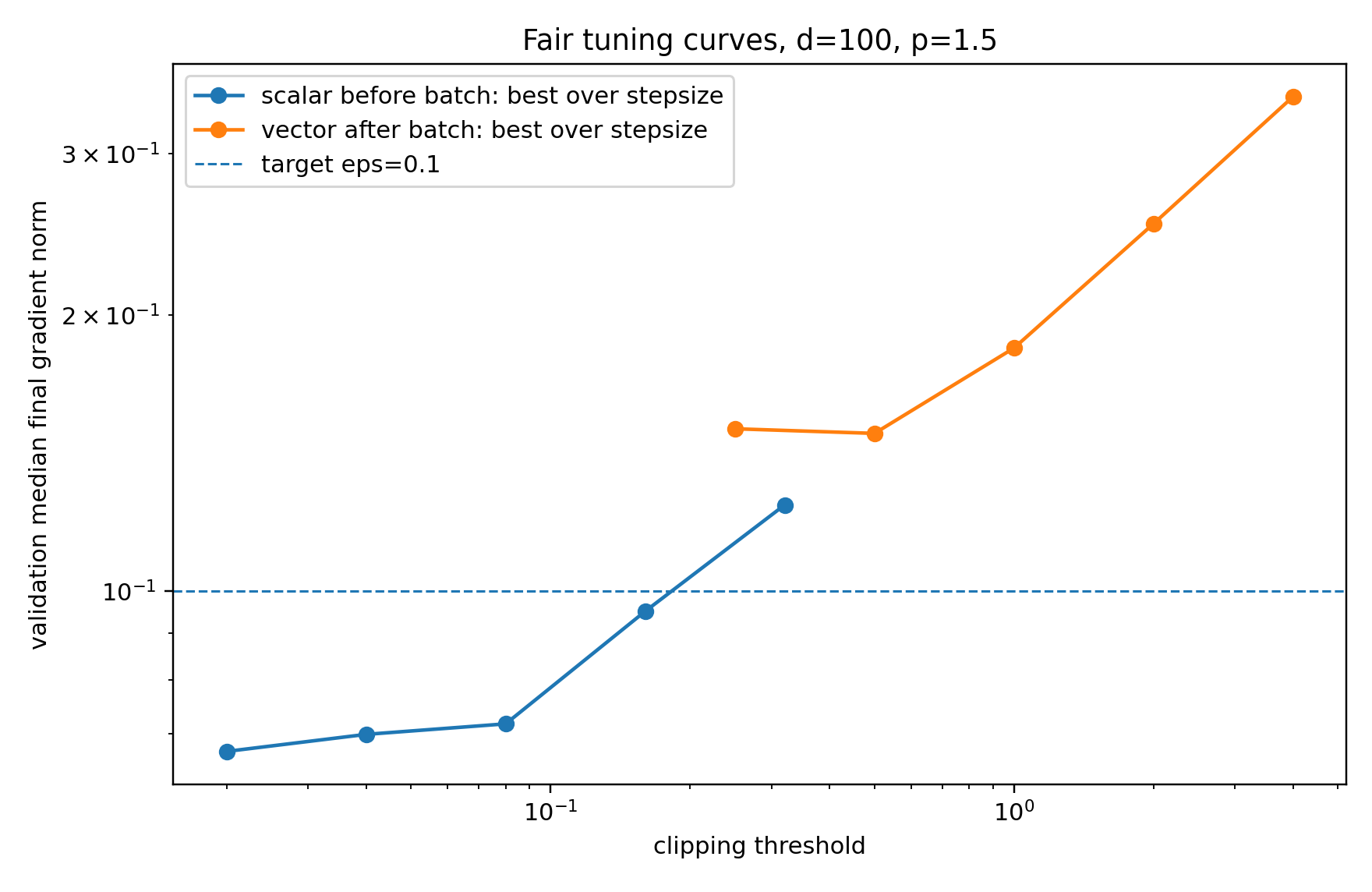}
    \caption{
    Tuning curves for the representative setting \(d=100\), \(p=1.5\).
    For each scalar threshold \(\tau\) and vector clipping radius \(r_{\rm vec}\),
    we report the best validation performance over stepsizes. Even after tuning,
    scalar clipping before aggregation achieves a lower final stationarity
    measure than final-vector clipping after aggregation.
    }
    \label{fig:fair-tuning}
\end{figure}

\paragraph{Representative optimization behavior.}
Figure~\ref{fig:representative-optimization} shows the optimization trajectory
for \(d=100\), \(p=1.5\), \(M=256\), and target accuracy \(\epsilon=0.1\).
The raw estimator is unstable under the infinite-variance noise. Final-vector
clipping stabilizes the iterates, but the method plateaus above the target
accuracy. In contrast, scalar clipping before aggregation continues to decrease
the stationarity measure and reaches the target reliably.

\begin{figure}[t]
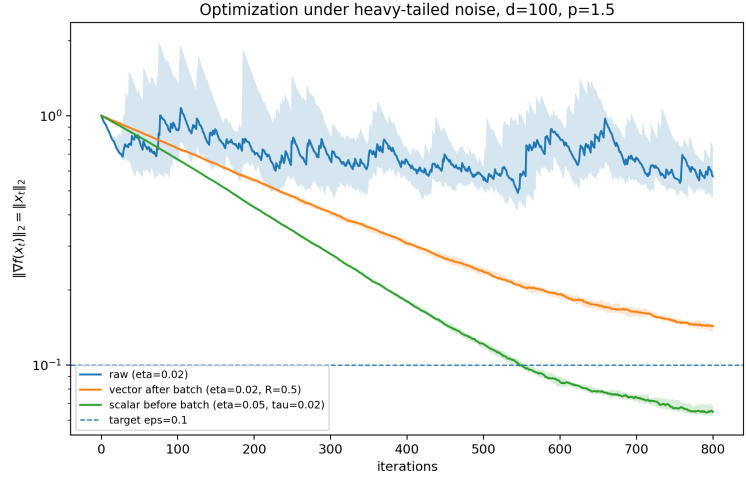

    \centering
    \safeincludegraphics[width=0.72\linewidth]{EXP/representative_optimization_curves.png}
    \caption{
    Representative optimization curves for \(d=100\), \(p=1.5\). We plot
    \(\|\nabla f(x_t)\|_2=\|x_t\|_2\) versus the number of iterations. Raw
    zeroth-order optimization is unstable under heavy-tailed noise. Final-vector
    clipping stabilizes the method but plateaus above the target. Scalar clipping
    before aggregation reaches the target and obtains the lowest final
    stationarity measure.
    }
    \label{fig:representative-optimization}
\end{figure}

The corresponding median final stationarity values are:
\[
\begin{array}{c|c|c}
\text{method} & \text{median final } \|\nabla f(x_T)\|_2 & \text{success rate} \\
\hline
\text{raw} & 0.662 & 0\% \\
\text{vector after batch} & 0.142 & 0\% \\
\text{scalar before batch} & 0.065 & 100\%
\end{array}
\]
Thus final-vector clipping helps significantly relative to the raw estimator,
but scalar clipping is the only method that reliably reaches the target in this
setting.

\paragraph{Direction diagnostic.}
To verify the mechanism, we measure the cosine similarity between each estimator
and the true gradient:
\[
    \frac{\langle g,\nabla f(x)\rangle}
    {\|g\|_2\|\nabla f(x)\|_2}.
\]
Figure~\ref{fig:direction-diagnostic} shows that the raw and final-vector-clipped
estimators have nearly identical cosine distributions. This is expected because
final-vector clipping only rescales \(g_{\rm raw}\) and therefore preserves its
direction. By contrast, scalar clipping substantially improves alignment with
the true gradient, increasing the median cosine from approximately \(0.317\) to
approximately \(0.543\).

\begin{figure}[t]
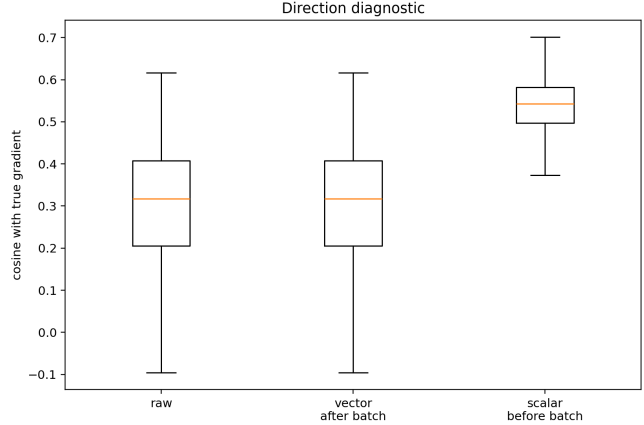

    \centering
    \safeincludegraphics[width=0.62\linewidth]{EXP/cosine_direction_diagnostic.png}
    \caption{
    Direction diagnostic for \(d=100\), \(p=1.5\). We report the cosine similarity
    between each stochastic estimator and the true gradient. Final-vector
    clipping has essentially the same alignment as the raw estimator because it
    preserves the raw direction. Scalar clipping before aggregation yields much
    better alignment, explaining its improved optimization performance.
    }
    \label{fig:direction-diagnostic}
\end{figure}

This confirms that the advantage of scalar clipping is not merely norm control.
The improvement comes from changing the aggregation itself: extreme directional
derivatives are controlled before they can determine the direction of the
batched estimator.

\paragraph{Dimension sweep.}
We next vary the ambient dimension \(d\) while fixing the tail exponent to
\(p=1.5\). Figure~\ref{fig:dimension-sweep} reports the median final
stationarity and the success rate. The raw estimator degrades rapidly with
dimension. Final-vector clipping improves stability but fails to reach the
target for larger dimensions. Scalar clipping remains below the target across
all tested dimensions.

\begin{figure}[t]
    \centering
    \begin{minipage}{0.48\linewidth}
        \centering
        \safeincludegraphics[width=\linewidth]{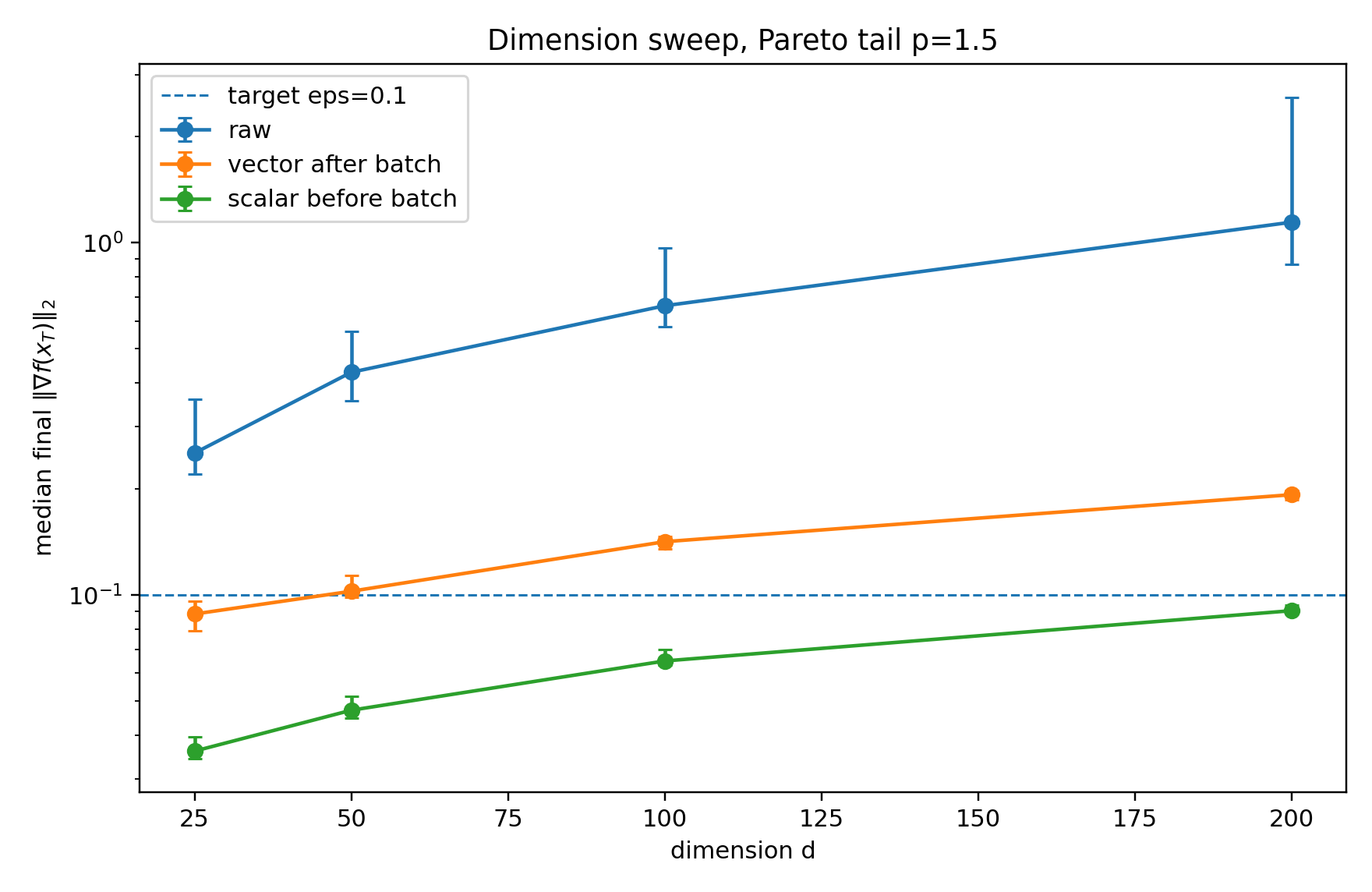}
    \end{minipage}
    \hfill
    \begin{minipage}{0.48\linewidth}
        \centering
        \safeincludegraphics[width=\linewidth]{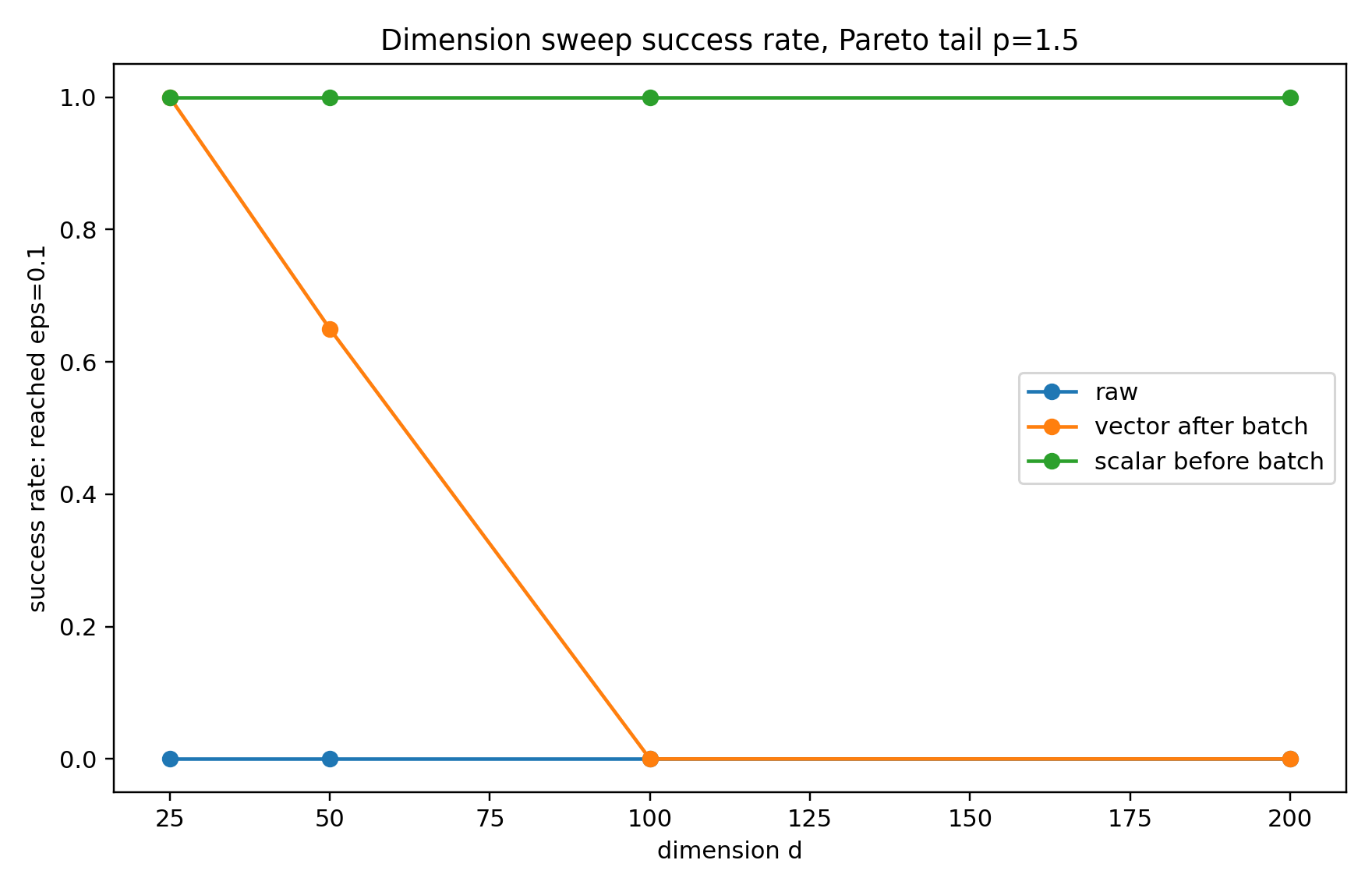}
    \end{minipage}
    \caption{
    Dimension sweep for Pareto tail exponent \(p=1.5\). Left: median final
    stationarity \(\|\nabla f(x_T)\|_2\). Right: success rate for reaching
    \(\epsilon=0.1\). Scalar clipping before aggregation remains robust as the
    dimension grows, while raw zeroth-order optimization and final-vector
    clipping degrade.
    }
    \label{fig:dimension-sweep}
\end{figure}

The median final stationarity values are:
\[
\begin{array}{c|ccc}
d
& \text{raw}
& \text{vector after batch}
& \text{scalar before batch}
\\
\hline
25  & 0.253 & 0.088 & 0.036 \\
50  & 0.429 & 0.102 & 0.047 \\
100 & 0.662 & 0.142 & 0.065 \\
200 & 1.144 & 0.193 & 0.090
\end{array}
\]
The success rates show the same trend: scalar clipping succeeds on all tested
dimensions, while vector clipping succeeds only at smaller dimensions and raw
zeroth-order optimization fails throughout.

\paragraph{Tail-index sweep.}
Finally, we vary the Pareto tail exponent \(p\) while fixing \(d=100\).
Smaller \(p\) corresponds to heavier tails. Figure~\ref{fig:tail-sweep} shows
that the advantage of scalar clipping is strongest in the heavy-tailed regime
and decreases as the noise becomes lighter-tailed.

\begin{figure}[t]
    \centering
    \begin{minipage}{0.48\linewidth}
        \centering
        \safeincludegraphics[width=\linewidth]{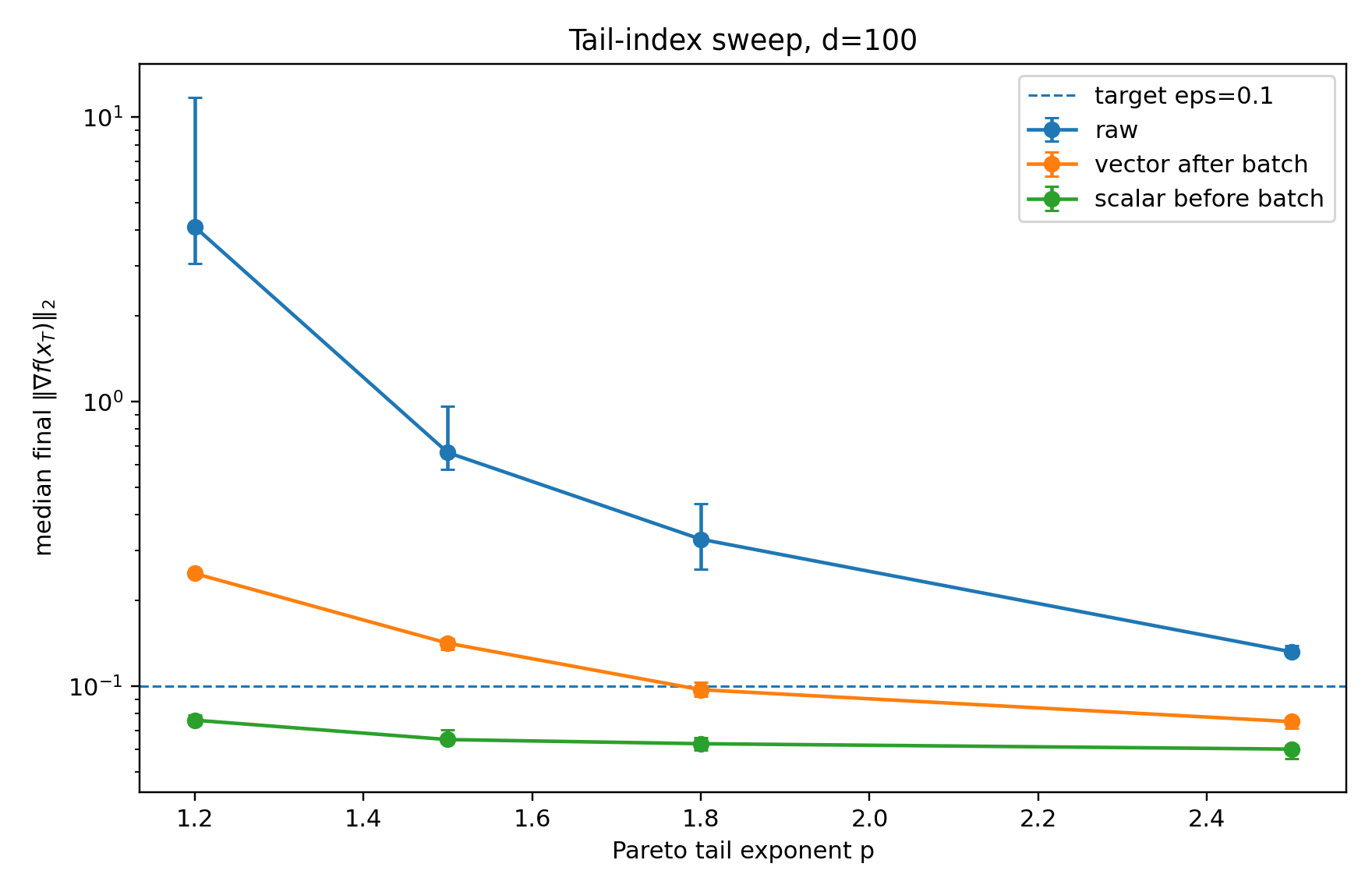}
    \end{minipage}
    \hfill
    \begin{minipage}{0.48\linewidth}
        \centering
        \safeincludegraphics[width=\linewidth]{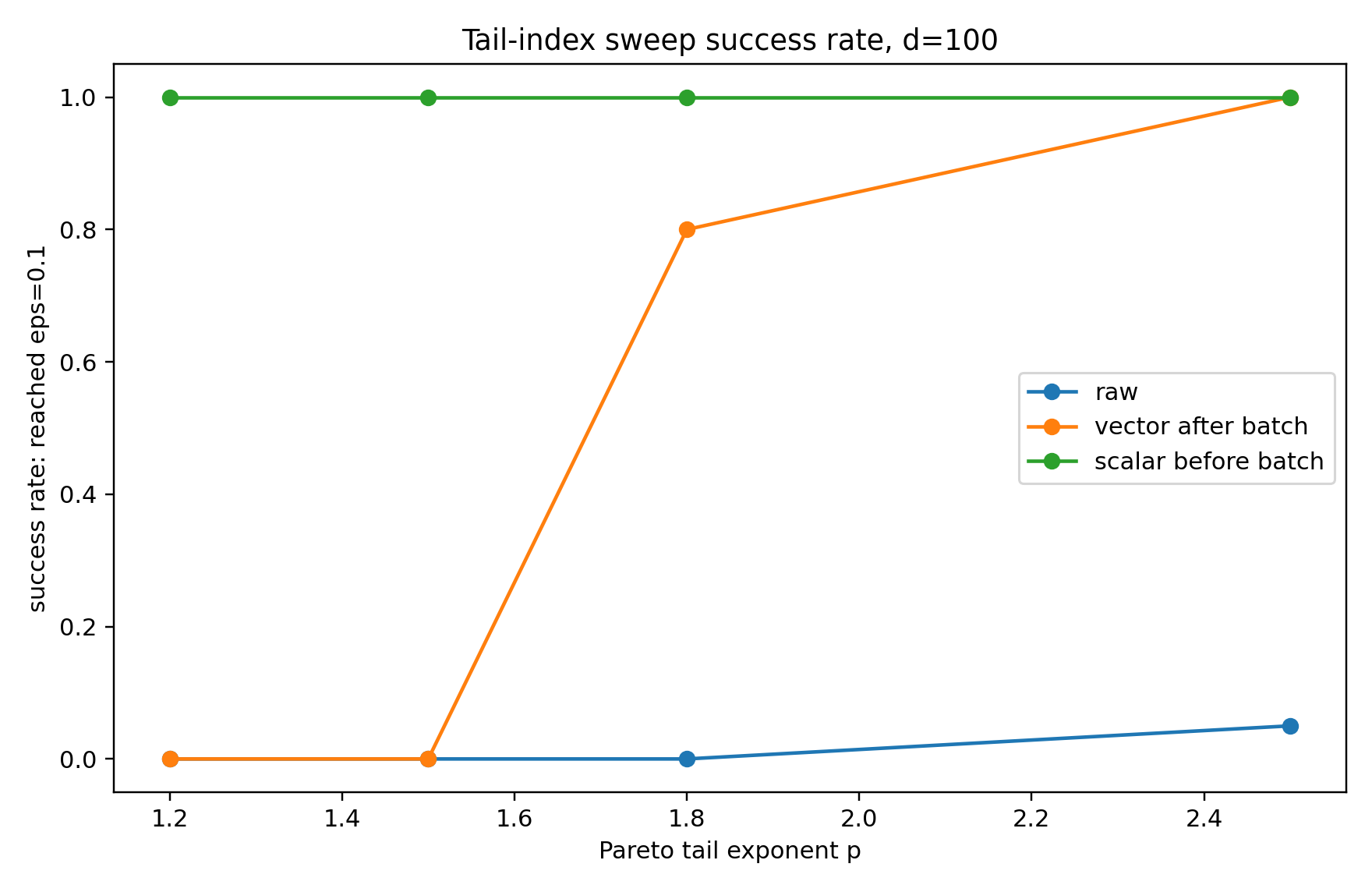}
    \end{minipage}
    \caption{
    Tail-index sweep for \(d=100\). Smaller \(p\) corresponds to heavier-tailed
    noise. Scalar clipping gives the largest improvement in the infinite-variance
    regime \(p<2\). We also include \(p=2.5\) as a finite-variance sanity check,
    where the gap becomes smaller but scalar clipping remains competitive.
    }
    \label{fig:tail-sweep}
\end{figure}

The median final stationarity values are:
\[
\begin{array}{c|ccc}
p
& \text{raw}
& \text{vector after batch}
& \text{scalar before batch}
\\
\hline
1.2 & 4.120 & 0.249 & 0.076 \\
1.5 & 0.662 & 0.142 & 0.065 \\
1.8 & 0.329 & 0.097 & 0.063 \\
2.5 & 0.132 & 0.075 & 0.060
\end{array}
\]
For \(p=1.2\), the raw estimator fails dramatically and final-vector clipping is
not sufficient to reach the target. Scalar clipping remains stable and achieves
a substantially lower stationarity measure. As \(p\) increases and the tails
become lighter, the gap narrows, as expected. The case \(p=2.5\) lies outside
our weak-\(L_p\), \(p\le 2\), theoretical regime and is included only as a
finite-variance sanity check.

\paragraph{Outlier dominance.}
We also record the ratio between the largest directional derivative and the
median directional derivative in each batch. Figure~\ref{fig:outlier-diagnostic}
shows that this ratio is often large, confirming that individual directional
samples can dominate the raw average under heavy-tailed noise. This supports the
need to clip before aggregation.

\begin{figure}[t]
    \centering
    \safeincludegraphics[width=0.62\linewidth]{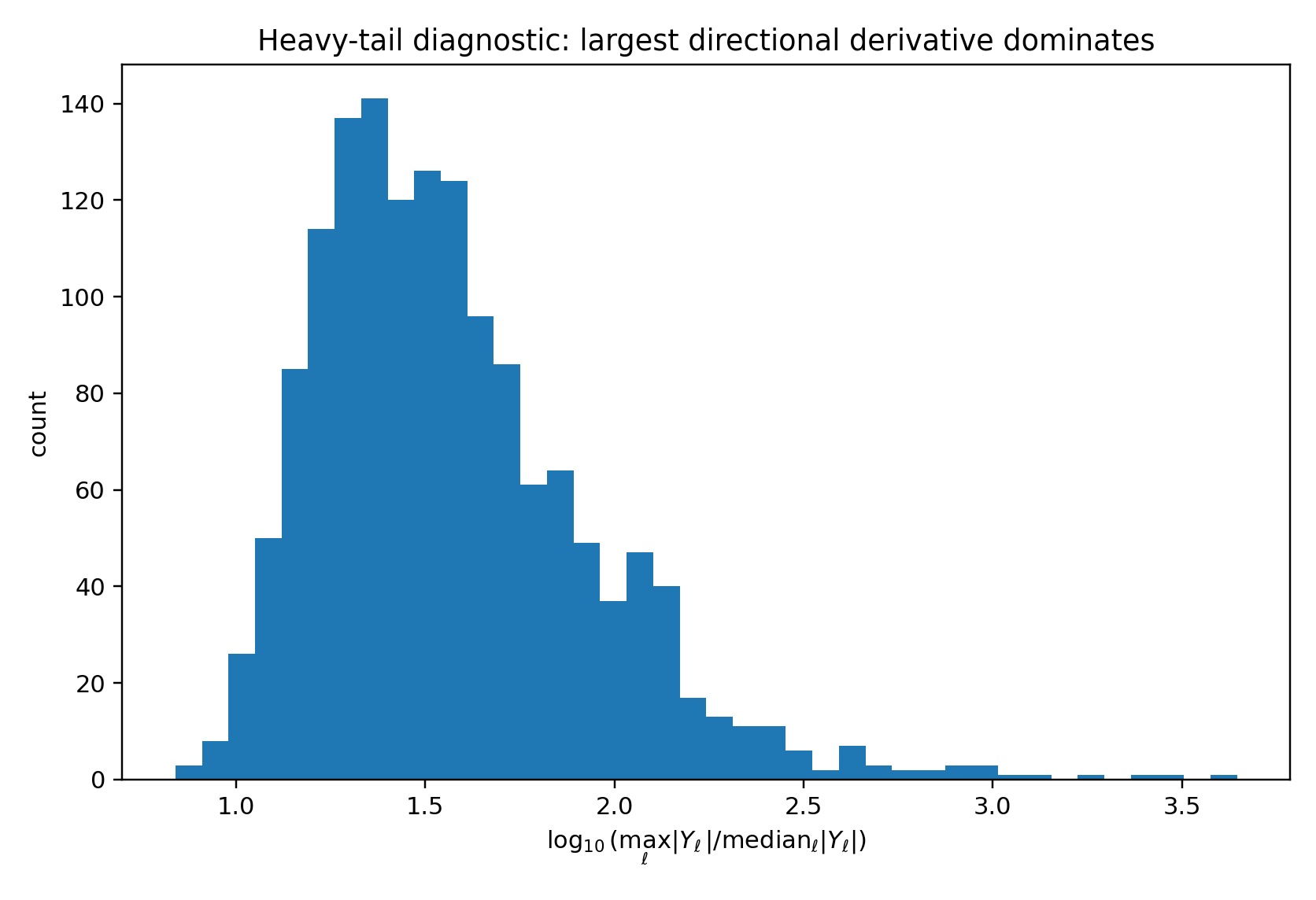}
    \caption{
    Outlier dominance diagnostic for the representative setting. We plot the
    distribution of
    \(\log_{10}(\max_\ell |Y_\ell|/\operatorname{median}_\ell |Y_\ell|)\).
    Large values indicate that a small number of directional derivatives can
    dominate the raw batched estimator. Scalar clipping prevents these outliers
    from controlling the aggregate direction.
    }
    \label{fig:outlier-diagnostic}
\end{figure}

\paragraph{Summary.}
These experiments support the estimator-level mechanism underlying our theory.
Final-vector clipping is useful because it controls the norm of the stochastic
zeroth-order estimate, and indeed it improves over the raw estimator. However,
it preserves the direction of the already-aggregated heavy-tailed estimate.
Scalar clipping acts earlier: it clips each noisy directional derivative before
aggregation. This yields better alignment with the true gradient, lower final
stationarity, and higher success rates across dimensions and tail indices.

\section{Additional Details for MeZO Experiments}
\label{app:mezo-details}

This appendix gives the implementation details for the RoBERTa-large fine-tuning
experiments in Section~\ref{sec:mezo-experiments}. We follow the prompt-based
MeZO setup of \citet{malladi2023mezo}. All methods are trained for 100K steps
on \(k=512\) examples per task. Standard deviations are computed over five
seeds.

For the multi-direction RSC-ZO variants, the estimator averages over
\(M\in\{1,2,4\}\) random directions. To keep the total forward-evaluation budget
fixed, the data batch size is scaled inversely with \(M\). Thus, increasing
\(M\) improves the directional estimate but reduces the number of examples used
per update.

\begin{table}[h]
\centering
\begin{tabular}{lc}
\toprule
\textbf{Hyperparameter} & \textbf{Value} \\
\midrule
Training steps & 100K \\
Base batch size & 16 \\
Learning rate & \(10^{-6}\) \\
Smoothing parameter \(\mu\) & \(10^{-3}\) \\
Weight decay & 0 \\
Number of seeds & 5 \\
\bottomrule
\end{tabular}
\vspace{6pt}
\caption{Shared hyperparameters for RoBERTa-large fine-tuning experiments. For
multi-direction RSC-ZO, the data batch size is scaled inversely with \(M\) to
keep the forward-evaluation budget fixed.}
\label{tab:hparams}
\end{table}

\begin{table}[h]
\centering
\begin{tabular}{lccc}
\toprule
& \textbf{SNLI} & \textbf{MNLI} & \textbf{TREC} \\
\midrule
Best \(\tau\) & 8 & 32 & 32 \\
\bottomrule
\end{tabular}
\vspace{6pt}
\caption{Best clipping threshold \(\tau\) for RSC-ZO on each task, selected by
grid search over \(\{2,4,8,16,32,64\}\).}
\label{tab:tau}
\end{table}

\end{document}